\documentclass[12pt]{article}

\usepackage{amsmath}
\usepackage{amsfonts}
\usepackage{amssymb}
\usepackage{amscd}
\usepackage{amsthm}
\usepackage{textcomp}
\usepackage{bbm}
\usepackage{color}
\usepackage[linktocpage=true,plainpages=false,pdfpagelabels=false]{hyperref}
\usepackage{graphics}

\input xypic

\usepackage[matrix,arrow,curve]{xy}

\usepackage{color}

\newcommand{\quash}[1]{}


\newtheorem{defin}{Definition}
\newtheorem{prop}{Proposition}
\newtheorem{nt}{Remark}
\newtheorem{Th}{Theorem}
\newtheorem{lemma}{Lemma}
\newtheorem{cons}{Corollary}

\newtheorem{defin-prop}{Definition-proposition}

\newfont{\ssdbl}{msbm8}
\newfont{\sdbl}{msbm9}
\newfont{\dbl}{msbm10 at 12pt}

\newcommand{\oo}{{\cal O}}
\newcommand{\ff}{{\cal F}}

\newcommand{\res}{\mathop {\rm res}}
\newcommand{\Hom}{\mathop {\rm Hom}}

\newcommand{\tr}{\mathop {\rm tr}}

\newcommand{\Spec}{\mathop {\rm Spec}}

\newcommand{\Id}{\mathop {\rm Id}}

\newcommand{\dz}{\mathbb{Z}}

\newcommand{\Z}{\dz}

\newcommand{\vsgm}{\varsigma}

\newcommand{\Ker}{{\rm Ker}\:}
\newcommand{\Image}{{\rm Im}\:}

\newcommand{\lrto}{\longrightarrow}

\def\Z{{\mathbb Z}}
\def\Q{{\mathbb Q}}

\newcommand{\vp}{\varphi}

\newcommand{\LL}{{\mathcal L}}

\newcommand{\AutL}{{   {\mathcal Aut}^{\rm c, alg} ({\mathcal L} )}}
\newcommand{\AutLm}{{   {\mathcal Aut}_{-}^{\rm c, alg} ({\mathcal L} )}}

\newcommand{\AutLp}{{   {\mathcal Aut}_{+}^{\rm c, alg} ({\mathcal L} )}}

\newcommand{\Lie}{\mathop {\rm Lie}}

\newcommand{\GG} {{\mathbb G}}
\newcommand{\CC}{\mathop{\rm CC}}
\newcommand{\GL}{{\mathop{\rm GL}}}

\newcommand{\ve}{\varepsilon}

\newcommand{\vpl}{{\mathbb V}_+}
\newcommand{\vmi}{{\mathbb V}_-}
\newcommand{\Nil}{{\mathop {\rm Nil}}}

\newcommand{\DistTo}{\xrightarrow{
   \,\smash{\raisebox{-0.65ex}{\ensuremath{\scriptstyle\sim}}}\,}}

\newcommand{\G}{{\mathcal G}}

\newcommand{\Der}{\mathop{ \rm Der}\nolimits^{\rm c}_A}

\newcommand{\qq}{{\mathcal Q}}

\newcommand{\mm}{{\mathcal M}}

\newcommand{\id}{{\rm id}}

\newcommand{\II}{{\mathcal I}}

\begin{document}

\author{
D. V. Osipov
}

\title{Local analog of the Deligne-Riemann-Roch isomorphism for line bundles in relative dimension $1$
\thanks{This work is supported by the Russian Science Foundation under grant no.~23-11-00033, https://rscf.ru/project/23-11-00033/}
}
\date{}

\maketitle

\begin{abstract}
We prove a local analog of the Deligne-Riemann-Roch isomorphism in the case of line bundles and  relative dimension $1$.
This local analog consists in computation of the class of $12$th power of the  determinant central extension of a group ind-scheme $\mathcal G$ by the multiplicative group scheme over $\Q$ via the product of $2$-cocyles in the second cohomology group.
These $2$-cocycles  are the compositions of the Contou-Carr\`{e}re symbol  with the $\cup$-product of $1$-cocycles.
  The group ind-scheme $\G$ represents the functor which assigns to every commutative ring $A$ the group that  is the semidirect product of the group
  $A((t))^*$ of invertible elements of $A((t))$  and the group of continuous $A$-automorphisms of  $A$-algebra $A((t))$.
  The determinant central extension naturally  acts on the determinant line bundle on the moduli stack
of geometric data (proper quintets). A proper quintet is a collection of a proper family of  curves over $\Spec A$, a line bundle on this family, a section of this family, a relative formal parameter at the section, a formal trivialization of the bundle at the section that satisfy  further conditions.
\end{abstract}

\section{Introduction}

The goal of this paper is to prove a local analog of the Deligne-Riemann-Roch theorem for
linear bundles in relative dimension $1$.
 The parts for this local analog of the Deligne-Riemann-Roch theorem consist of  the central extensions of the  group ind-scheme
 that is the semidirect product of the group
 of invertible functions on the  formal punctured disc  and the group of automorphisms of this disc. These central extensions are by the multiplicative group scheme $\GG_m$.
 We prove an equivalence of these central extensions after extensions of scalars  of all schemes and ind-schemes from the ring $\dz$ to
 the field
 $\Q$.

\subsection{Deligne-Riemann-Roch theorem}  \label{Deligne}
We recall the Deligne-Riemann-Roch isomorphism (or, in other words, the Deligne-Riemann-Roch theorem).

Let $\pi : X \to S$ be a  smooth proper morphism of relative dimension~$1$ of schemes, i.e. $\pi$ gives a  family of smooth proper curves over a scheme $S$.

For any two invertible sheaves $L$ and $M$ on $X$ Deligne constructed (see~\cite[\S~6]{D1}, \cite[Expos\'e~XVIII, \S~1.3]{SGA4}) an invertible sheaf  $\langle L, M \rangle$ on $X$. This Deligne pairing (or, in other words, Deligne bracket) is symmetric and bilinear with respect to the tensor products of invertible sheaves. It is also functorial in $L$ and $M$ (with respect to isomorphisms of sheaves) and compatible with base change.

When  $S$ is a smooth algebraic variety over a field, then   inside the group ${\rm Pic}(S)$ the Deligne bracket for invertible sheaves $L$ and $M$ on $X$ is
\begin{equation}  \label{DBr1}
c_1(\langle L, M \rangle) = \pi_* (c_1(L)  \cdot c_1(M))  \, \mbox{,}
\end{equation}
where $\cdot$ is the product in the Chow ring  ${\rm CH}^*(X)$, and $\pi_* : {\rm CH}^2(X)  \to {\rm CH}^1(S) ={\rm Pic}(S) $  is the direct image between Chow groups.

Let $\omega = \Omega^1_{X/S}$ be an invertible sheaf of relative differential $1$-forms on $X$, and $L$ be any invertible sheaf on $X$.
When geometric fibres of $\pi$ are connected,  Deligne constructed a  Riemann-Roch isomorphism of line bundles (see~\cite[Th\'eor\`eme~9.9]{D1}):
\begin{equation}\label{DRR1}
(\det R\pi_* L)^{\otimes 12} \simeq \langle L, L  \rangle^{\otimes 6}  \otimes \langle L, \omega  \rangle^{\otimes -6} \otimes \langle  \omega, \omega \rangle  \mbox{.}
\end{equation}
Isomorphism~\eqref{DRR1} is functorial in $L$ and compatible with base change.

When  $S$ is a smooth algebraic variety over a field, then in  ${\rm Pic}(S) \otimes_{\dz} \Q$, using~\eqref{DBr1},  isomorphism~\eqref{DRR1} becomes an equality
\begin{multline} \label{GRR}
c_1(\det R\pi_* L) = \pi_* \left( \frac{1}{2} c_1(L)  \cdot c_1(L) - \frac{1}{2} c_1(L) \cdot c_1(\omega)  + \frac{1}{12} c_1(\omega) \cdot c_1(\omega) \right) = \\  =\pi_* \left( \left( {\rm ch}(L) \cdot {\rm td(\omega^{\otimes -1})}     \right)_2  \right) \, \mbox{,}
\end{multline}
where $${\rm ch}(L) = 1 + c_1(L) + \frac{1}{2} c_1^2(L) + \ldots$$
is the Chern character,
$${\rm td} (\omega^{\otimes -1}) = 1 - \frac{1}{2} c_1(\omega) + \frac{1}{12} c_1^2(\omega) + \ldots$$ is the Todd class, and $(\hphantom{m})_2$ means the component of degree $2$ in the ring ${\rm CH}^*(X)  \otimes_{\dz} \Q$, i.e. an element from the group ${\rm CH}^2(X) \otimes_{\dz} \Q$.

Thus we see that equality~\eqref{GRR} is the corollary of the Grothendieck-Riemann-Roch theorem.

\subsection{Invertible functions and automorphisms of the punctured formal disc}

We consider the  functor  $L \GG_m $ from the category of commutative rings to the category of Abelan groups
$$
A \longmapsto A((t))^*  \, \mbox{,}
$$
where $A((t))^*$ is the group of invertible elements in the $A$-algebra of Laurent series $A((t))= A[[t]][t^{-1}]$.
This functor is represented by a group ind-scheme, see~Section~\ref{ex-ind-sch}.

The group ind-scheme $L \GG_m $ can be considered as the group of invertible functions on the formal punctured disc.

On the $A$-algebra $A((t))$ there is the natural   $t$-adic topology  that makes $A((t))$ into the topological $A$-algebra, where $A$ has discrete topology.

We consider the functor $\AutL$  from the category of commutative rings to the category of  groups that assigns to every commutative ring $A$ the group of all continuous $A$-automorphisms of the $A$-algebra $A((t))$.   This is indeed a functor, and this functor is represented by a group ind-scheme, see~Sections~\ref{sec-aut} and~\ref{ex-ind-sch}.

The group ind-scheme $\AutL $ can be considered as the group of automorphisms of  the formal punctured disc.

For any commutative ring $A$ there is the natural action of the group $\AutL(A)$ on the Abelian group $A((t))^*$.
Therefore it makes sense the group ind-scheme
$$
\G = L \GG_m \rtimes \AutL  \, \mbox{.}
$$

This group ind-scheme has the following geometric meaning. There is a natural action of $\G$ on the moduli stack of quintets $\mathcal M$ and on the moduli stack of proper quintets ${\mathcal M}_{\rm pr}$ (see Theorem~\ref{action}), and we explain now the notion of a quintet and of a proper quintet.

By a quintet over a commutative ring $A$ we mean a collection that consists of a separated  family $C$ of  curves over $A$, a section of this family such that this family is smooth near this section, a sheaf of $\oo_C$-modules that is an invertible sheaf near the section, a relative formal parameter at the section, a formal trivialization of the sheaf at the section (in other words, the last two conditions mean that the topological $A$-algebra of  functions on the formal neighbourhood of $C$ at the section   is isomorphic to the topological $A$-algebra  $A[[t]]$ and the sheaf obtained as the restriction of the sheaf of $\oo_C$-modules to this formal neighbourhood is trivial, and  we fix an element $t$ and a formal trivialization of the sheaf), see Section~\ref{quint}.

A proper quintet is a quintet such that, in addition to the previous conditions, the morphism $C \to \Spec A$ is a
a flat proper finitely presented morphism such that all geometric fibers
are integral one-dimensional schemes, and the sheaf  is an invertible sheaf of $\oo_C$-modules, see also Section~\ref{quint}.

Informally speaking, elements of $\G(A)$ reglue the family of curves  and the sheaf in a quintet  along the section (cf.~\cite[\S~17.3, \S~18.1.3]{FB}, \cite[\S~4.1]{Pol}, but we provide detailed proofs of the construction, see Theorem~\ref{action}).
The action of $\G$ on a proper quintet with a smooth curve $C$ over a field $k$ is transitive, i.e. the map  from the corresponding Lie algebra to the tangent space is surjective.

We note that in complex analytic category some analog
of functor of proper quintets modulo  isomorphisms  (when all fibers of $C$ from a quintet are smooth curves of genus~$g$) is
represented by an infinite-dimensional complex manifold, see~\cite{ADKP}.
Such quintets are related with the infinite-dimensional Sato Grassmanian and the Krichever map that maps quintets over fields to the points of the  Sato Grassmanian,
see, e.g.~\cite{ADKP, SW}. Through the Krichever map there is also further relationships with the moduli spaces of curves, soliton equations etc.

\subsection{Local analog of the left hand side of  the Deligne-Riemann-Roch isomorphism: determinant central extension}

There is a natural  central extension, which we call the determinant central extension, of group ind-schemes
$$
1 \lrto \GG_m \lrto \widetilde{\G}  \stackrel{\eta}{\lrto} \G \lrto 1  \, \mbox{,}
$$
where the morphism $\eta$ admits a section (as ind-schemes, not as group ind-schemes).

The construction of the determinant central extension is as follows (see also Section~\ref{main-centr}).

Let $A$ be any commutative ring. The group $\G(A)$ naturally acts on $A((t))$. Let $(h, \vp) \in L \GG_m(A) \rtimes \AutL(A) $ and $f \in A((t))$.
Then $(h, \vp) (f)= h \vp(f)$.

The group $\widetilde{\G}(A)$ consists of pairs $(g,s)$, where $g \in \G(A)$ and an element $s$
belongs to  the free $A$-module  $\det(g (A[[t]])  \mid  A[[t]])$ of rank $1$
and also generates  this $A$-module.
Here $\det(g(A[[t]])  \mid  A[[t]])$  is the relative determinant of $A$-modules $g(A[[t]])$ and $A[[t]]$, which is  canonically isomorphic to the $A$-module\footnote{Here and further in the article all the wedge products are over $A$. For short, we will omit this indication.}
$$
\Hom\nolimits_A \left( \bigwedge^{\rm max} (g(A[[t]])/ t^l A[[t]]), \,  \bigwedge^{\rm max} (A[[t]]/ t^l A[[t]])                   \right)      \, \mbox{,}
$$
where an integer $l$ satisfies the condition  $t^l A[[t]]   \subset (g(A[[t]]) \cap A[[t]])$.
Now $\eta((g,s)) =g$.

We note that this construction  originates (when $A=k$ is a field) from~\cite{KP}, where the group $\G(A)$ is changed to the group of continuous automorphisms of topological $k$-vector space $k((t))$.

A (non-group) section of morphism $\eta$ can be chosen in a  natural way, see Remark~\ref{can-sec}.

On the moduli stack of proper quintets ${\mathcal M}_{\rm pr}$ there is a natural determinant linear bundle, see Section~\ref{det-lin-bund}.
This means that for any proper quintet with a family of curves $\pi : C \to \Spec A$ and an invertible sheaf of $\oo_C$-modules  $\ff$
there is a linear bundle on $\Spec A$ that satisfy further compatibilty conditions. This linear bundle on $\Spec A$ is exactly an invertible sheaf  $\det R \pi_* \ff$
from Section~\ref{Deligne}.

In Proposition~\ref{cohom} we construct an explicit  complex of finitely generated projective $A$-modules, which consists of two terms and whose determinant gives $\det R \pi_* \ff$. Using this complex,  we construct a natural action of the group ind-scheme $\widetilde{\G}$ on the determinant line bundle on the moduli stack of proper quintets ${\mathcal M}_{\rm pr}$ that lifts an action of the group ind-scheme $\G$
on  ${\mathcal M}_{\rm pr}$, see Theorem~\ref{act-det-st}.

We can consider cohomology groups of group ind-schemes with coefficients in commutative group ind-schemes. More generally,  instead of ind-schemes we consider  functors from the category of commutative rings, see Section~\ref{cohom-sec}.  Then we write an analog of the bar complex (or, in other words, an analog of the standard complex) from the group cohomology,
where we will use the Abelian groups of morphisms of functors,
 see complex~\eqref{cochain-functor1}. Thus for any integer $q \ge 0$ we obtain the notions of $q$-cocycles, $q$-coboundaries, $q$th cohomology groups of a group functor with coefficients in a commutative group functor, where functors are from the category of commutative rings.

We outline also another approach to compute these cohomology groups. This approach is by resolution by induced modules (certain commutattive group functors) that are acyclic, see~Section~\ref{induced}.

Similar to group cohomology, we have the notion of $\cup$-products of cocycles. Besides, any central extension of a group functor by a commutative group functor that admits a  section (as functors) defines a $2$-cocycle. The equivalence classes of such extensions are in one-to-one correspondence with elements of the second cohomology group.

For any group functor $G$ that acts on a commutative group functor $F$ and any commutative ring $A$ there is a natural homomorphism of Abelian groups
$$
H^q(G, F)  \lrto H^q(G(A), F(A))  \, \mbox{.}
$$

We will denote by $D$ the $2$-cocycle on $\G$ with coefficients in $\GG_m$ given by the determinant central extension and its natural section.

The determinant central extension (and the $2$-cocycle $D$) give the local analog of the left hand side of the Deligne-Riemann-Roch isomorphism.

\subsection{Local analog of the right hand side  of the Deligne-Riemann-Roch isomorphism: $\cup$-products of $1$-cocycles}

There is the Contou-Carr\`{e}re symbol $\CC$
$$
\CC \; : \; L \GG_m  \times L \GG_m   \lrto L \GG_m \otimes  L \GG_m  \lrto \GG_m  \, \mbox{.}
$$
that is a bimultiplicative and antisymmetric  morphism
(see Section~\ref{CC-symb} on the definition of the Contou-Carr\`{e}re symbol $\CC$, some its properties and the references).
The Contou-Carr\`{e}re symbol $\CC$ is invariant under the diagonal action of the group functor $\AutL$.

Through the natural morphism of group ind-schemes  $\G \to \AutL  $ we obtain that the group ind-scheme $\G$ naturally acts on the commutative group ind-scheme $L \GG_m$.

Let $\lambda_1$ and $\lambda_2$ be any $1$-cocycles on $\G$ with coefficients in $L \GG_m$. We construct the $2$-cocycle
$\langle \lambda_1 , \lambda_2   \rangle$ on $\G$ with coefficients in  $\GG_m$ (where $\G$ acts trivially on $\GG_m$)
$$
\langle \lambda_1 , \lambda_2   \rangle  = \CC \circ (\lambda_1 \cup \lambda_2)  \, \mbox{,}
$$
where $\circ$ means the composition of the morphism $\lambda_1  \cup  \lambda_2$ from $\G \times \G$  to $L \GG_m \otimes L \GG_m  $ and
 the $\G$-equivariant morphism  $\CC$ from $ L \GG_m \otimes  L \GG_m $ to  $\GG_m $.

 There are distinct $1$-cocycles $\Lambda$ and $\Omega$ on $\G$ with coefficients in $L \GG_m$ (see also Section~\ref{prod-coc}):
$$
\Lambda((h, \vp)) = h    \qquad \mbox{and}  \qquad \Omega( (h, \vp) )  = d \vp(t) / dt = \vp(t)'  \, \mbox{,}
$$
where $(h, \vp)  \in \G(A) = L \GG_m (A)  \rtimes \AutL (A)$.

The $1$-cocycle $\Lambda$ is the universal $1$-cocycle (see Remark~\ref{univers})  that will replace $L$ in the right hand side of formula~\eqref{DRR1}, and the $1$-cocycle $\Omega$ will replace $\omega$ in the right hand side of formula~\eqref{DRR1}.

Now we obtain distinct $2$-cocycles on $\G$ with coefficients in $\GG_m$ (where $\G$ acts trivially on $\GG_m$):
$$
\langle \Lambda , \Lambda   \rangle   \, \mbox{,} \qquad \quad \langle \Lambda , \Omega   \rangle
 \, \mbox{,} \qquad \quad  \langle \Omega , \Omega   \rangle  \, \mbox{.}
$$
Here the first $2$-cocycle is a local analog of $\langle L , L \rangle $ from formula~\eqref{DRR1},
the second $2$-cocycle is a local analog of $\langle L , \omega \rangle $ from~\eqref{DRR1},
and the third $2$-cocycle is a local analog of  $\langle \omega , \omega \rangle $ from~\eqref{DRR1}. Besides, the Contou-Carr\`{e}re symbol $\CC$ is an analog of the direct image $\pi_*$, see formula~\eqref{DBr1}.

\subsection{Equality in the second cohomology group}

We obtain in Theorem~\ref{LDRR} the local analog of the Deligne-Riemann-Roch isomorphism~\eqref{DRR1}
$$
D^{12} =  \langle \Lambda, \Lambda \rangle^6 \cdot    \langle \Lambda, \Omega \rangle^{-6} \cdot \langle \Omega, \Omega \rangle \, \mbox{,}
$$
where this equality is in the group $H^2 (\G_{\Q}, {\GG_m}_{\Q})$,  and ${\G}_{\Q}$ acts trivially on   ${\GG_m}_{\Q}$. Besides we extend scalars from $\dz$ to $\Q$ for all the schemes and ind-schemes. We use also  the multiplicative notation for the group law in the Abelian
group $H^2 (\G_{\Q}, {\GG_m}_{\Q})$.

We note that we stated
Theorem~\ref{LDRR}
without proof in the short note~\cite{O2}.

The proof of the above equality in $H^2 (\G_{\Q}, {\GG_m}_{\Q})$ is based on the  statement from Corollary~\ref{Cor2}. This statement  generalizes Theorem~5.1  from~\cite{O1} from the case of the group ind-scheme ${\AutL}_{\Q}$ to the group ind-scheme ${\G}_{\Q}$. Besides,   for the case of the group of orientation-preserving   diffeomorphisms of the circle in the theory of infinite-dimensional Lie groups see Corollary (7.5) from~\cite{Se}.

The statement of Corollary~\ref{Cor2} is that an element from the group $H^2 (\G_{\Q}, {\GG_m}_{\Q})$,
where ${\GG_m}_{\Q}$
is a trivial ${\G}_{\Q}$-module, is uniquely defined by its image in $H^2(\Lie \II \G^0_{\Q}(\Q), \Q )$ together with its restriction
to $H^2(  {\G_+}_{\Q}, {\GG_m}_{\Q})$.

Here $\G_+$ is the group ind-scheme that represents the functor that assigns to every commutative ring $A$ the group which is the semidirect product of the group $A[[t]]^*$ of invertible elements of the $A$-algebra $A[[t]]$ and the group of continuous $A$-automorphisms
of the $A$-algebra $A[[t]]$. We have that $\G_+ (A)  \subset \G(A)$.

The Lie $\Q$-algebra $\Lie \II \G^0_{\Q}(\Q)$ is the $\Q$-subalgebra of Lie $\Q$-algebra  $\Lie \G(\Q)$ of the group ind-scheme $\G_{\Q}$. The algebra $\Lie \G (\Q)$  is the algebra of continuous differential operators of order $\le 1$ acting on $\Q((t))$. The subalgebra $\Lie \II \G^0_{\Q}(\Q)$  is
$$
\Q[t,t^{-1}]  \rtimes  \mathop{{\rm Der}}\nolimits_{\Q} (\Q[t, t^{-1}])  \, \mbox{,}
$$
where $ \mathop{{\rm Der}}_{\Q} (\Q[t, t^{-1}])  =  \Q[t, t^{-1}] \frac{\partial}{\partial t}  $ is the Lie $\Q$-algebra of $\Q$-derivations.

Now to obtain the local analog of the Deligne-Riemann-Roch isomorphism we note that the $2$-cocycles in the left and right hand sides are trivial
after restriction to ${\G_+}_{\Q}$. And for the comparison of the corresponding Lie algebra $2$-cocyles it is enough to check the corresponding equalities for  elements of type $t^n$ and $t^m  \frac{\partial}{\partial t}$  from $\Lie \II \G^0_{\Q}(\Q)$.

We note that this check is reduced to the finite sum of constant elements in case of $\langle \Lambda, \Lambda \rangle$, to the sum of finite arithmetic progression in case of $\langle \Lambda, \Omega \rangle$, and to the sum of quadratic function over a finite set of consecutive non-negative integers  in case of $\langle \Omega, \Omega \rangle$,
see the proof of Theorem~\ref{main-Lie}.

\subsection{Organization of the paper}

The paper is organized as follows.

In Section~\ref{Sec-Loop-Aut} we introduce and give some properties for the group functors $L \GG_m$, $\AutL$, $\G$.

In Section~\ref{coh-gr-fun} we introduce and prove properties for cohomology of group functors.

In Section~\ref{Lie-sect} we investigate the Lie algebra valued functor $\Lie \G$, in particularly for any commutative ring $A$ the Lie $A$-algebra $\Lie \G (A)$, constructed from the ind-scheme~$\G$.

In Section~\ref{main-centr}  we give various constructions of the determinant central extension of $\G$ and prove its properties.

In Section~\ref{geom-act} we define quintets and proper quintets and describe the action of $\G$ on the moduli stack of quintets $\mathcal M$ and on the moduli stack of proper quintets ${\mathcal M}_{\rm pr}$, and describe also the action of $\widetilde{\G}$ on the determinant line bundle on ${\mathcal M}_{\rm pr}$.

In Section~\ref{2-cocycles}  we define and investigate $2$-cocycles $\langle \Lambda, \Lambda \rangle $,
$\langle \Lambda, \Omega \rangle $ and $\langle \Omega, \Omega \rangle $ on $\G$ with coefficients in $\GG_m$.

In Section~\ref{compar} we prove a local analog of the Deligne-Riemann-Roch isomorphism by the comparison of two central extensions of ${\G}_{\Q}$ by ${\GG_m}_{\Q}$.

\section{Groups related with the ring of Laurent series} \label{Sec-Loop-Aut}

Let $A$ be any commutative ring. Let $A((t))= A[[t]][t^{-1}]$ be the ring of Laurent series over $A$. We will denote $\LL (A) = A((t))$.

By a group functor (or commutative group functor) we mean a covariant functor from the category of commutative rings to the category of groups (or Abelian groups).

\subsection{Loop functor of multiplicative group}   \label{recall-loop}

As usual, let $\GG_m$ be a commutative group functor such that $\GG_m(A)= A^*$ for any commutative ring $A$.

By $L \GG_m$ we denote the group loop functor of $\GG_m$, i.e.
$$L \GG_m(A) = \GG_m(A((t)))= A((t))^* \mbox{.} $$

There is a canonical  isomorphism of commutative group functors (see~\cite[Lemma~1.3]{CC1}, \cite[Lemma~0.8]{CC2})
\begin{equation}  \label{md}
\underline{\dz} \times \GG_m \times \vpl  \times \vmi \, \simeq \,  L \GG_m \mbox{,}
\end{equation}
where the group functors in the left hand side and the corresponding embeddings are described as follows.

Let $A$ be any commutative ring.

The group  $\underline{\dz}(A)$ is the group of locally constant $\dz$-valued functions on $\Spec A$. Any $\underline{n} \in \underline{\dz}(A)$ gives the decomposition $A = A_1 \times \ldots \times A_l$ such that
the function $\underline{n}$ restricted to any $\Spec A_i$ is the constant function with value $n_i \in \dz$.
Then $\underline{n}  \mapsto t^{\underline{n}} = t^{n_1} \times \ldots \times t^{n_l}$ defines embedding $\underline{\dz} \to L \GG_m$
of the corresponding group functors in~\eqref{md}.

The group embedding  $\GG_m(A)  \hookrightarrow L \GG_m(A)$ is given by mapping of an invertible element from $A$ to the series consisting of only a constant term.

The subgroups $\vpl(A)$ and $\vmi(A)$ of the group $L \GG_m(A) $  are defined as
\begin{gather}  \label{eq-vpl}
\vpl(A) = \left\{ \left. 1 + \sum_{l > 0} a_l t^l  \quad \right|  \quad   a_l \in A     \quad  \mbox{for any} \quad l > 0         \right\}   \\
\label{eq-min}
\vmi(A) = \left\{   1 + \sum_{l < 0} a_l t^l \quad \left|  \quad  \sum_{l < 0} a_l t^l  \in A((t)) \, \mbox{,} \; a_l \in \Nil(A)    \right.        \quad \mbox{for any} \quad l< 0           \right\} \, \mbox{,}
\end{gather}
where $\Nil(A)$ is the nil-radical of $A$, i.e. the set of all nilpotent elements of $A$.

We denote
\begin{equation}  \label{declg}
(L \GG_m)^0 =   \GG_m \times \vpl  \times \vmi  \, \mbox{.}
\end{equation}
Thus we obtain decomposition of group functors
\begin{equation}  \label{dec}
\underline{\dz} \times (L \GG_m)^0 \, \simeq \,  L \GG_m  \, \mbox{.}
\end{equation}
By $\nu : L \GG_m  \to \underline{\dz}$ we denote the  morphism of group functors given by the corresponding  projection in decomposition~\eqref{dec}.

\subsection{Automorphism group and semidirect product}  \label{sec-aut}

On the ring $A((t))$, where $A$ is any commutative ring,  there is the natural topology with  the base of neighbourhoods of zero consisting of $A$-submodules
$U_n = t^n A[[t]]$, $n \in \dz$. This topology makes the ring $A((t))$ into a topological ring.

By $\AutL (A)$ we denote the group of all $A$-automorphisms of the $A$-algebra $A((t))$ that are homeomorphisms.

We have the following facts on elements of this group (see~\cite[\S~2.1]{O1}  and references therein). Any continuous $A$-automorphism of the $A$-algebra $A((t))$ is a homeomorphism. There is the following isomorphism of sets:
\begin{equation}  \label{isom}
\AutL (A) \,  \simeq  \,   \left\{  \left. h \in L \GG_m(A)   \,  \right|  \,  \nu(h) =1   \right\} \, \mbox{,}
\end{equation}
where the map from the left side to the right side is $\vp  \mapsto \widetilde{\vp}= \vp(t)$, and the map from the right side to the left side is
$$
\widetilde{\vp} \longmapsto  \left\{ f \mapsto  f \circ \widetilde{\vp} \right\} \, \mbox{,}
$$
where  $f \in A((t))$, and  $f \circ \widetilde{\vp}$ denotes the series from $A((t))$ obtained by substitution of the series $\widetilde{\vp}$ into the series $f$
instead of variable $t$.

For any $\vp_1$ and $\vp_2$ from $\AutL (A)$ we have
$$
\widetilde{\vp_1 \vp_2} = \widetilde{\vp_2}  \circ \widetilde{\vp_1}  \, \mbox{.}
$$

Isomorphism~\eqref{isom} defines the structure of the  functor on the correspondence
$$A \longmapsto \AutL (A)  \, \mbox{,}$$
where $A$ is any commutative ring.
This is the group functor which we denote by $\AutL$.

We will use the following unique decomposition of the functor  $\AutL$ (see~\cite[\S~3.1]{O1}):
\begin{equation}  \label{decaut}
\AutL = \AutLp  \cdot \AutLm =  \AutLm  \cdot \AutLp   \, \mbox{,}
\end{equation}
where for any commutative ring $A$ the subgroups $\AutLp(A)$ and $\AutLm(A)$ of the group $\AutL(A)$ are described via isomorphism~\eqref{isom} as the set of elements
\begin{equation}  \label{apl}
\left\{ \left. \sum_{l \ge 0} a_l t^l  \quad \right| \quad   a_l \in A \quad \mbox{for any} \quad l \ge 0 \,  \mbox{,}    \quad  a_0 \in \Nil(A) \, \mbox{,}  \quad a_1 \in A^*   \right\}
\end{equation}
and
\begin{equation} \label{ami}
\left\{  \left.
\sum_{n \le l \le -1} a_l t^l  + t  \quad \right| \quad n < 0 \quad \mbox{is any} \, \mbox{,} \quad  a_l \in \Nil(A) \quad \mbox{for any } \quad l < 0
\right\}
\end{equation}
correspondingly.
We note that the group $\AutLp(A)$ consists of all continuous $A$-automorphisms of the $A$-algebra $A[[t]]$.

 Besides, decomposition~\eqref{decaut} is the direct product decomposition as functors, but it is not the direct product or semi-direct product decomposition as group functors (the corresponding subgroups are not normal for some rings $A$).

\bigskip

We have the natural action of the group functor $\AutL$ on the commutative group functor $L \GG_m$. This leads to the following group functor that will play one of the main roles in this article
\begin{equation}  \label{G-funct}
{\mathcal G} = L \GG_m \rtimes \AutL  \mbox{,}
\end{equation}
where for any commutative ring $A$, for any elements $h_1, h_2 \in L \GG_m(A)$ and any elements $\vp_1, \vp_2 \in \AutL(A)$ we have
$$
(h_1, \vp_1)(h_2, \vp_2)= (h_1 \vp_1(h_2), \vp_1 \vp_2)  \, \mbox{.}
$$

Besides the group functor ${\mathcal G}$ acts on the commutative group functor $\LL$ in the natural way:
$(h,\vp) (f) = h \vp(f)$, where $h \in L\GG_m(A)$, $\vp \in \AutL(A)$, $f \in \LL(A)=A((t))$ for any commutative ring $A$.

We note that the group ${\mathcal G}(A)$ acts by $A$-module homeomorphisms on $A((t))$ for any commutative ring $A$.

\bigskip

We note that for any commutative ring $A$, any $\vp  \in  \AutL(A)$  and any ${h \in L \GG_m(A)}$  we have $\nu (\vp(h)) = \nu(h)$.

Therefore the group functor  $\AutL$ preserves the group subfunctor \linebreak ${(L \GG_m)^0  \hookrightarrow L \GG_m}$ under the natural action. Hence  the group subfunctor $\G^0  \hookrightarrow \G$ defined by
\begin{equation}  \label{decg0}
{\mathcal G}^0 = (L \GG_m)^0 \rtimes \AutL  
\end{equation}
is well-defined.

\begin{lemma}  \label{dec-G}
We have $\G = \G^0  \rtimes \underline{\dz} $, where $\underline{\dz} \hookrightarrow L \GG_m \hookrightarrow \G$ as in Section~\ref{recall-loop}.
\end{lemma}
\begin{proof}
 For any commutative ring $A$,  any $h \in (L \GG_m)^0(A)$ and any ${\vp \in  \AutL(A)}$ we have inside the group $\G(A)$ the following equality
$$
 (t,1) (h, \vp) (t, 1)^{-1} =  (t \, h, \vp)(t^{-1}, 1)  = (t \, h  \, \vp(t)^{-1}, \vp)
\, \mbox{,}
$$
and $(t \, h \, \vp(t)^{-1} , \vp ) \in \G^0(A)$, since $\nu(t \, h \, \vp(t)^{-1})=0$.
\end{proof}

\subsection{Ind-schemes}  \label{ex-ind-sch}

By an ind-affine ind-scheme we mean an ind-object ${M= \mbox{``$\varinjlim\limits_{i \in I}$''} \Spec C_i}$  of the category of affine schemes such that all transition maps in the ind-object are closed embeddings of schemes. We say that ind-affine ind-scheme is ind-flat if any $C_i$ is a flat $\dz$-module.

By $\oo(M)= \mbox{$\varprojlim\limits_{i \in I} $} C_i$ we will denote the  ring of regular functions of $M$.

We note that all the  functors described in Sections~\ref{recall-loop}-\ref{sec-aut} are represented by ind-flat ind-affine ind-schemes.

Indeed (see formulas~\eqref{eq-vpl}-\eqref{eq-min}, \eqref{apl}-\eqref{ami}),
\begin{gather}  \label{ind-formb}
\GG_m = \Spec \dz[x,x^{-1}] \, \mbox{,}  \qquad \quad \underline{\dz} = \mbox{``$\varinjlim\limits_{n \in \mathbb{N}}$''} \coprod_{-n \le l \le n}
(\Spec \dz)_l \, \mbox{,} \\
\vpl  = \Spec \dz[a_1, a_2, \ldots]  \, \mbox{,}  \qquad  \quad \vmi = \mbox{``$\varinjlim\limits_{ \{   \epsilon_i \} }$''}  \Spec \dz[a_{-1}, a_{-2}, \ldots]/ I_{\{ \epsilon_i \}}  \, \mbox{,} \\
\AutLp =  \mbox{``$\varinjlim\limits_{n \in \mathbb{N}}$''}  \Spec \dz[a_0]/ (a_0^n)  \times  \Spec \dz[a_1, a_1^{-1}]  \times \Spec \dz[a_2, a_3, \ldots ]  \, \mbox{,}  \\
\label{ind-formc}
\AutLm =  \mbox{``$\varinjlim\limits_{ \{   \epsilon_i \} }$''}  \Spec \dz[a_{-1}, a_{-2}, \ldots]/ I_{\{ \epsilon_i \}}  \, \mbox{,}
\end{gather}
where $(\Spec \dz)_l = \Spec \dz $ and the limits in $\vmi$ and $\AutLm$  are taken over all the sequences $\{\epsilon_i\}$ with integers $i < 0$ and $\epsilon_i$ are nonnegative integers such that all but finitely many $\epsilon_i$ equal zero, the ideal $I_{\{\epsilon_i\}}$ is generated by elements $a_i^{\epsilon_i +1}$ for all integers $i < 0$.

The other functors $L \GG_m$, $(L \GG_m)^0$, $\AutL$, $\G$, $\G^0$ are represented by products of some above ind-schemes.

\section{Cohomology of group functors}  \label{coh-gr-fun}
\subsection{General definitions and properties}  \label{cohom-sec}
We fix a commutative ring $R$. We say that a functor is over $R$ if this functor is from the category of commutative $R$-algebras.
All functors (correspondingly, group or commutative group functors) in this section will be over $R$. We will omit this indication on~$R$.

For any functor $G$ and any commutative $R$-algebra $R'$ we denote by $F_{R'}$ the restriction of the functor $F$ to the category of commutative $R'$-algebras.

For any  functor $G$ and any commutative group functor $F$ we denote by $\Hom(G, F)$ the Abelian group of all morphisms from the functor $G$ to the functor $F$. In this definition we do not take into account any additional structure (if it exists) on the functor $G$.

For any integer $n \ge 1$ and any functor $G$ we denote by $G^{\times n}$ the functor which is the $n$-fold direct product of the functor $G$.

A group functor $G$  acts on a commutative group functor $F$  if for any commutative $R$-algebra $A$ the group $G(A)$ acts on the Abelian group $F(A)$ such that this action is compatible with any $R$-algebra homomorphism of commutative $R$-algebras. We will also say that $F$ is a {\em $G$-module}.

\medskip

Now we recall the definition and the main properties (see~\cite[\S~2.3.1]{O1}) for cohomology of group functors.

We will fix a group functor $G$ until the end of this section.

Let  $F$ be a $G$-module. Then there is the cochain complex of Abelian groups (this complex is similar to the complex obtained from the bar or standard resolution in the  group cohomology)

\begin{equation}  \label{cochain-functor1}
C^0 (G,F)  \stackrel{\delta_0}{\lrto} C^{1}(G, F) \stackrel{\delta_1}{\lrto} \ldots \xrightarrow{\delta_{k-1}} C^{k}(G,F) \stackrel{\delta_k}{\lrto} \ldots  \, \mbox{,}
\end{equation}
where $C^0(G, F)= F(R)$, $C^k(G, F)= \Hom (G^{\times k}, F)$ when $k \ge 1$, and
\begin{multline*}
\mathop{\delta_{q}c} \, (g_1, \ldots, g_{q+1}) = g_1  \mathop{c}  (g_2, \ldots, g_{q+1}) \, \cdot \,  \prod_{i=1}^q  \mathop{c}  (g_1, \ldots, g_i g_{i+1}, \ldots, g_{q+1})^{(-1)^i} \,  \cdot\\
\cdot \,   \mathop{c}  (g_1, \ldots, g_q)^{(-1)^{q+1}}  \, \mbox{,}
\end{multline*}
where integers $q \ge 0$, $c \in C^q(G,F)$, $g_j \in G(A)$ with $1 \le j \le q+1$ for any commutative $R$-algebra $A$. Besides, we use the multiplicative notation for the group laws in  Abelian groups $F(A)$ and $C^{q}(G,F)$.

Let $q$ be  a non-negative  integer.  A {\em $q$-cocycle} on $G$ with coefficients in $F$ is an element of the subgroup $\Ker \delta_q \subset  C^{q}(G,F)$. For $q \ge 1$,
a {\em $q$-coboundary} on $G$ with coefficients in $F$ is an element of the subgroup $\Image \delta_{q-1} \subset  C^{q}(G,F)$.
The  {\em  $q$-th cohomology group} of $G$ with coefficients in $F$ is the Abelian group
$$H^q(G,F)= \Ker \delta_q / \Image \delta_{q-1}  \, \mbox{.}$$

To give a  $q$-cocycle on $G$ with coefficients in $F$ is the same as to give a collection of $q$-cocycles $\{ c_A \}$ on  $G(A)$
with coefficients in $F(A)$ for any commutative $R$-algebra $A$ with obvious compatibility condition for $c_{A_1}$ and $c_{A_2}$ for any $R$-algebra homomorphism $A_1 \to A_2$. Besides, for any fixed commutative $R$-algebra $A$ there is a natural morphism from complex~\eqref{cochain-functor1} to the corresponding complex that calculates the cohomology of $G(A)$ with coefficients in $F(A)$. This gives the natural homomorphism
$$
H^q(G, F)  \lrto H^q(G(A), F(A))  \, \mbox{.}
$$

Any morphism  of $G$-modules $F_1 \to F_2$ induces the homomorphism of Abelian groups  $H^q(G,F_1) \to H^q(G, F_2)$.

We can speak about a short exact sequence (or an extension) of group functors, this is equivalent that this sequence  becomes  a short exact sequence of groups after taking the values of  group functors on any commutative $R$-algebra $A$.

By~\cite[Prop.~2.2]{O1}, the elements of the group $H^2(G,F)$ are in one-to-one correspondence with equivalence classes
 of  extensions of the group functor $G$ by the commutative  group functor $F$
\begin{equation}  \label{ext1}
1 \lrto F \lrto \widetilde{G} \stackrel{\vartheta}{\lrto} G \lrto 1
\end{equation}
such that the morphism $\vartheta$ has a section $G \to \widetilde{G}$ as a morphism of functors (in general,   not as a morphism of group functors), and
the action of $G$ on $F$ comes from innner automorphisms in the group functor $\widetilde{G}$. Besides, as usual, extension~\eqref{ext1} is central if and only if $G$ acts trivially on $F$.

As in the group cohomology (see~\cite[ch.~V, \S~3]{Bro}), for any $1$-cocycles $\lambda_1$ and $\lambda_2$ on $G$ with coefficients in $G$-modules $F_1$ and $F_2$ correspondingly, we obtain a $2$-cocycle $\lambda_1 \cup \lambda_2$ on $G$ with coefficients in $F_1 \otimes F_2$ such that
$$
(\lambda_1 \cup  \lambda_2) (g_1, g_2) = \lambda_1 (g_1)  \otimes g_1 ( \lambda_2(g_2) ) \, \mbox{,}
$$
where for any commutative $R$-algebra $A$ we have $(F_1 \otimes F_2)(A) = F_1(A) \otimes_{\dz} F_2(A)$, and $g_1, g_2 $ are any elements from $G(A)$.
This induces the well-defined homomorphism, which is called the $\cup$-product:
$$
\cup \; :  \; H^1(G, F_1)  \otimes_{\Z} H^1(G, F_2) \lrto H^2(G, F_1 \otimes F_2)   \, \mbox{.}
$$

\subsection{Resolution by induced modules}   \label{induced}

We will explain that it is also  possible to calculate $H^q(G,F)$ via the resolution of $F$ by induced $G$-modules, which are acyclic.

For any $G$-module  $F$ we denote the group
$F^{G}= H^0(G, F) $
which consists of elements $f \in F(R)$ such that for any commutative $R$-algebra $A$ and any element $g \in G(A)$ we have $g(f)=f$.
Now $F \mapsto F^{G}$ is a left exact functor from the Abelian  category of $G$-modules to the category of Abelian groups.

Let $L$ be any commutative group functor. By $\underline{\Hom}(G, L)$ we denote a $G$-module defined as
$$
\underline{\Hom}(G, L)(A) = \Hom(G_A, L_A)  \, \mbox{,}
$$
where $A$ is any commutative $R$-algebra, and the $G$-module structure is the following:
for any $g \in G(A)$, $h \in \underline{\Hom}(G, L)(A)$, $x \in G(A')$, where $A'$ is any commutative $A$-algebra, we define
$g(h)(x)= h(x g)$.

There is an embedding of $G$-modules
\begin{equation}  \label{emb}
\alpha \; : \; \xymatrix{ L \; \, \ar@{^{(}->}[r] & \, \underline{\Hom}(G, L)}
\end{equation}
given for any commutative $R$-algebra $A$ as $L(A) \ni m \mapsto \alpha(m) \in \underline{\Hom}(G, L)(A)$,
where $\alpha(m)(x)= x m$ for any $x \in G(A')$ and any commutative $A$-algebra $A'$.

Besides, there is a morphism of group functors (in general, not a morphism of $G$-modules)
\begin{equation}  \label{spl}
\beta \; : \; \underline{\Hom}(G, L)   \lrto L  \qquad \mbox{such that} \qquad \beta \alpha = \Id  \, \mbox{,}
\end{equation}
where for any commutative $R$-algebra $A$, any $h \in \underline{\Hom}(G, L)(A)$ we have $\beta(h)= h(e)$, where $e$ is the identity element of the group $G(A)$.

We will call also the $G$-module $\underline{\Hom}(G,L)$ as an {\em induced $G$-module}.

\begin{prop}
We fix a group functor $G$.
\begin{enumerate}
\item Any short exact sequence
$$
 \xymatrix{
 F_1 \; \ar@{^{(}->}[r] & F_2 \;  \ar@{->>}[r]^{\gamma} & F_3
}
$$
of $G$-modules such that $\gamma$ admits a section as functors (not necessarily as group functors)  induces the long exact sequence of Abelian groups
$$
\ldots \lrto H^q(G, F_1) \lrto H^q(G, F_2)  \lrto H^q(G, F_3) \lrto H^{q+1}(G, F_1) \lrto \ldots \, \mbox{.}
$$
\item For any commutative group functor $L$ and any integer $q > 0$ we have $$H^q(G, \underline{\Hom}(G, L)  ) =0 \mbox{.}$$
\item For any $G$-module $F$,   taking via~\eqref{emb} the resolution of $F$ by induced modules, then  applying to this resolution the functor $K \mapsto K^{G}$ (where $K$ is a $G$-module), and then taking the cohomology groups of this complex, we obtain the groups $H^*(G,F)$.
\end{enumerate}
 \end{prop}
\begin{proof}
{\em 1.} The proof  follows from a long  exact sequence of cohomology groups induced by a short exact sequence of cochain complexes~\eqref{cochain-functor1} for $F_1$, $F_2$, $F_3$  that is obtained from short exact sequences of Abelian groups
$$
1 \lrto C^k(G, F_1)  \lrto C^k(G, F_2) \lrto C^k(G, F_3) \lrto 1  \, \mbox{,}
$$
which follow because of a section $F_3 \to F_2$ as funtors.

{\em 2.} For complex~\eqref{cochain-functor1}
with $F =  \underline{\Hom}(G, L)$ we construct a contracting homotopy for any $q \ge 1$ (cf.~\cite[Expose I, Lemme~5.2.2.]{SGA3})
$$
s_q \; : \;  C^{q}(G,\underline{\Hom}(G, L))  \lrto C^{q-1}(G,\underline{\Hom}(G, L))
$$
by the rule:
$$
(s_q h) (g_1, g_2, \ldots, g_{q-1})(g)= h(g, g_1, \ldots, g_{q-1})(e)  \, \mbox{,}
$$
where $h \in C^q(G, \underline{\Hom}(G, L))$, $g_1, \ldots, g_{q-1} \in G(A)$ for a commutative $R$-algebra $A$, $g \in G(A')$ for any commutative $R$-algebra $A'$, $e \in G(A')$ is  is the identity element of the group $G(A')$.
The necessary equalities $\delta_{q-1} s_q + s_{q+1} \delta_q = \Id$ follow by direct calculation.

{\em 3.} This is a standard statement which follows from two previous items and splitting of the resolution into short exact sequences such that each of them has splitting as functors in virtue of~\eqref{spl}.
\end{proof}

\section{Lie algebra valued functors from group functors}

\label{Lie-sect}

We suppose that a group functor $G$ over a commutative ring $R$ is represented by an ind-affine ind-scheme over $R$. Then (see, e.g.,~\cite[Appendix~A]{O1}) we have the tangent space functor $\Lie G $ at the unit of the group $G(R)$:
$$
\Lie G(A) = \Ker \left( G(A[\ve]/ (\ve^2)) \lrto G(A) \right)  \, \mbox{,}
$$
where $A$ is any commutative $R$-algebra.

The functor $\Lie G$ is a commutative group functor over $R$ which has a natural structure of ${\mathbb A}^1_R$-module, where ${\mathbb A}^1_R$ is a ring functor over $R$ such that  ${\mathbb A}^1_R(A) = A$.
And there is a bracket
$$
  [ \cdot ,  \cdot ] \, : \, \Lie G  \times \Lie G  \lrto \Lie G
$$
which defines on $\Lie G (A)$ the structure of the Lie $A$-algebra  for any commutative $R$-algebra $A$ (i.~e. $\Lie G$ is a Lie algebra valued
functor over $R$).

For the  definition of $[\cdot, \cdot]$,  we take  elements $d_i \in \Lie G (A)  \subset  G (A[\ve_i]/ (\ve_i^2)) $, where $i=1$ and $i=2$.
 Then we obtain that $d_1, d_2 \in G(A[\ve_1, \ve_2] / (\ve_1^2, \ve_2^2))$ via the natural embedding
 $G (A[\ve_i]/ (\ve_i^2)) \hookrightarrow G(A[\ve_1, \ve_2] / (\ve_1^2, \ve_2^2))$, where $i=1$ and $i=2$. Now we have that
 \begin{equation}  \label{bra}
[d_1, d_2] = d_1 d_2 d_1^{-1}  d_2^{-1}  \in  G (A[\ve_1, \ve_2] / (\ve_1^2, \ve_2^2))
\end{equation}
is the image of the element from $\Lie G (A)$  under the embedding induced by the embedding of rings
$A [\ve_1 \ve_2] / (\ve_1^2 \ve_2^2)   \hookrightarrow A [\ve_1, \ve_2]/ (\ve_1^2, \ve_2^2)$, where we consider
$$
\Lie G(A) = \Ker \left( G(A[\ve_1 \ve_2]/ (\ve_1^2 \ve_2^2)) \lrto G(A) \right)  \, \mbox{.}
$$

The correspondence $G \mapsto \Lie G$ is a functor from the category of group functors over $R$ represented by group ind-affine ind-schemes over $R$ to the category of Lie algebra valued
functors over $R$.

\medskip

We consider the group functor $\G^0$ (see~\eqref{decg0}), which is represented by an ind-affine ind-scheme.

\begin{prop}  \label{prop-iso-Lie}
Let $A$ be any commutative ring. Then we have a natural isomorphism of Lie $A$-algebras
\begin{equation}  \label{iso-Lie}
\Lie \G^0(A) \simeq  A((t)) \rtimes \Der (A((t))) \, \mbox{,}
\end{equation}
where $A((t))$ is an Abelian Lie $A$-algebra and  $\Der (A((t)))$ is the Lie $A$-algebra of continuous $A$-derivations on the commutative $A$-algebra $A((t))$ with the commutator bracket.
This means that  $\Lie \G^0(A)$ is  $A((t))  + A((t)) \frac{\partial}{\partial t}$ as an $A$-module with the bracket
\begin{equation}  \label{Lie-br}
\left[ s_1 + r_1   \frac{\partial}{\partial t}, s_2 + r_2 \frac{\partial}{\partial t} \right]  = \left(r_1 s_2' - r_2 s_1' \right) +  \left(r_1 r_2' - r_2 r_1' \right)\frac{\partial}{\partial t}  \, \mbox{,}
\end{equation}
where $s_i $ and $r_i$ are from $A((t))$, and $r_i'$ or $s_i'$ means the derivative with respect to $t$ of the corresponding element.

In other words,  the Lie $A$-algebra $\Lie \G^0 (A)$ is naturally isomorphic to the Lie algebra of continuous differential operators of order $\le 1$ acting on functions
on the punctured disc over $A$:
\begin{equation}  \label{diff-operat}
\left(s + r   \frac{\partial}{\partial t} \right) \left(f \right) = sf + rf' \, \mbox{,} \qquad \mbox{where} \quad s,r,f \in A((t))  \, \mbox{.}
\end{equation}
\end{prop}
\begin{proof}
Let $A$ be a commutative ring.

By definition, $\Lie \G^0 (A) $  consists of elements $(1 + s \ve, \vp) \in \G^0 (A[\ve]/ (\ve^2))$, where ${s \in A((t))}$
and $\vp \in \AutL (A[\ve]/ (\ve^2))$ such that $\vp(t)= t + r \ve $ with $r \in A((t))$.

 For any $f  \in A((t))$ we have the action of $(1 + s \ve, \vp)$ on $f$ in $(A[\ve]/ (\ve^2))((t))$:
 $$
 (1 + s \ve, \vp)(f)= (1 + s\ve) \vp(f)= (1 + s \ve) \cdot (f  \circ (t + r \ve)) = (1 + s \ve)(f + r f' \ve)= f + (sf + rf') \ve  \, \mbox{.}
 $$
 Hence we obtain that the map $(1 + s \ve, \vp)  \mapsto s + r \frac{\partial}{\partial t} $  gives isomorphism~\eqref{iso-Lie} as $A$-modules and formula~\eqref{diff-operat}.

 To obtain the Lie bracket~\eqref{Lie-br} we use formula~\eqref{bra}. For the bracket between elements $r_1   \frac{\partial}{\partial t}$ and
  $r_2   \frac{\partial}{\partial t}$ see~\cite[\S~4.1]{O1}. The bracket between elements $s_1$ and $s_2$ is zero, because they come from the commutative group functor $(L \GG_m)^0$. Now we  calculate $[r_1   \frac{\partial}{\partial t}, s_2]$.  The answer follows from the following calculations in   $(A[\ve_1, \ve_2] / (\ve_1^2, \ve_2^2))((t))$:
  $$
  (1 + s_2 \ve_2) \circ (t + r_1 \ve_1) = 1+ (s_2 \circ (t + r_1 \ve_1)) \ve_2 = 1 + (s_2 + r_1 s_2' \ve_1) \ve_2 = (1+ r_1s_2' \ve_1 \ve_2)(1 + s_2 \ve_2)  \, \mbox{,}
  $$
  since in the group $\G^0(A[\ve_1, \ve_2]/ ( \ve_1^2, \ve_2^2))$ we have
  $$
  \vp_1 \cdot (1 + s_2 \ve_2) \cdot \vp_1^{-1} =  (1 + s_2 \ve_2) \circ (t + r_1 \ve_1)  \, \mbox{,}
  $$
  where $\vp_1 \in \AutL (A[\ve_1, \ve_2]/ (\ve_1^2, \ve_2^2))$
such that $\vp_1(t) = t + r_1  \ve_1$.
\end{proof}

\begin{nt} \label{tng-Lie} \em
We consider the group functor $\G$ (see~\eqref{G-funct}), which is represented by ind-affine ind-scheme. Clearly, from definition, for any commutative ring $A$ we have
$$
\Lie \G (A) = \Lie \G^0 (A)  \, \mbox{.}
$$
\end{nt}

\bigskip

We will need the following lemma, which is a punctured disc over $A$ analog of~\cite[Lemma~4.3]{ADKP}  in the holomorphic case.

\begin{lemma}  \label{commutant}
Let $A$ be a commutative ring such that $\frac{1}{2}  \in A$. Then we have
$$
\left[  \Lie \G^0(A) , \Lie \G^0(A)     \right] = \Lie \G^0(A)  \, \mbox{.}
$$
\end{lemma}
\begin{proof}
We note that $\left[  \Lie \G^0(A) , \Lie \G^0(A)     \right]$ is an $A((t))$-module, because
$$
f \left[ s_1 + r_1   \frac{\partial}{\partial t}, s_2 + r_2 \frac{\partial}{\partial t} \right] =
\left[  f r_1  \frac{\partial}{\partial t},  s_2 + \frac{1}{2} r_2 \frac{\partial}{\partial t}   \right] +
\left[   s_1 +  \frac{1}{2} r_1   \frac{\partial}{\partial t},  f r_2 \frac{\partial}{\partial t}    \right]   \, \mbox{,}
$$
where $f, r_i, s_i \in A((t))$ for $i=1$ and $i=2$. Therefore it is enough to show that elements $1$ and $\frac{\partial}{\partial t}$ belong to
$\left[  \Lie \G^0(A) , \Lie \G^0(A)     \right]$. We have
$$
\frac{\partial}{\partial t} = \left [  \frac{\partial}{\partial t}, t \frac{\partial}{\partial t}         \right]  \qquad \mbox{and}
\qquad
1 = \left [  \frac{\partial}{\partial t}, t        \right]  \, \mbox{.}
$$
\end{proof}

\section{Determinant central extension of~$\G$}

\label{main-centr}

We construct the determinant central extension of the group functor $\G$ (see~\eqref{G-funct})  by the group functor $\GG_m$.

\subsection{Block matrix}
We consider the group functor $\G^0$ (see~\eqref{decg0}).
We will use  the following decomposition into the direct sum.

\begin{prop}  \label{dir-sum-new}
Let $A$ be a commutative ring. For any $g \in \G^0(A)$  we have
\begin{equation}  \label{dir-sum}
A((t))= t^{-1}A[t^{-1}]  \oplus g( A[[t]])  \, \mbox{.}
\end{equation}
\end{prop}
\begin{proof}
We note that if an element $g_1$ belongs to the subgroup $\GG_m(A) \times \vpl (A)$ or to the subgroup $\AutLp(A)$  of the group $\G^0(A)$,
then the action of the element $g_1$ restricted to the $A$-submodule $A[[t]]$ is an isomorphism of this submodule.

If  an element $g_2$ belongs to the subgroup $\GG_m(A) \times \vmi(A)$ or to the subgroup $\AutLm(A)$  of the group $\G^0(A)$,
then the action of the element $g_2$ restricted to the $A$-submodule $t^{-1}A[t^{-1}]$ is an isomorphism of this submodule.

From decompositions~\eqref{decg0}, \eqref{declg} and~\eqref{decaut}  we have that for any $g \in \G^0(A)$ there is a  decomposition
$$
g = h_-  \,  h_+  \, \vp_-  \,  \vp_+  \, \mbox{,}
$$
where elements $h_-  \in \vmi(A)$, $h_+ \in \GG_m(A) \times  \vpl(A)$, $\vp_- \in \AutLm(A)$ and ${\vp_+ \in \AutLp(A)}$.

We have that an element $\vp_-^{-1} \, h_+  \, \vp_-$ belongs to the subgroup $(L \GG_m)^0(A)$. We denote this element by $\tilde{h}$.
According to~\eqref{declg}, we have a decomposition $\tilde{h} = \tilde{h}_-  \, \tilde{h}_+ $,
where elements $\tilde{h}_-  \in \vmi(A) $ and $\tilde{h}_+  \in \GG_m(A) \times  \vpl(A)$.  Hence we obtain
$$
g = h_- \,  \vp_-  \,  \vp_-^{-1}  \,  h_+  \,  \vp_-  \,  \vp_+  = h_- \,  \vp_-  \,  \tilde{h}  \, \vp_+  =  h_- \, \vp_-  \,  \tilde{h}_-  \,  \tilde{h}_+  \,  \vp_+ = p_- \, p_+  \, \mbox{,}
$$
where $p_- =  h_- \, \vp_-  \,  \tilde{h}_-$ and $p_+ = \tilde{h}_+  \,  \vp_+$. Besides,  the action of the element $p_-$ restricted to the $A$-submodule $t^{-1}A[t^{-1}]$ is an isomorphism of this submodule, and the action of the element $p_+$ restricted to the $A$-submodule $A[[t]]$ is an isomorphism of this submodule.

Now we obtain
\begin{multline*}
A((t))  = p_- (A((t))) = p_- ( t^{-1}A[t^{-1}]  \oplus  A[[t]]) =  p_- ( p_-^{-1} (t^{-1}A[t^{-1}])  \oplus  A[[t]]) = \\
= t^{-1}A[t^{-1}]  \oplus p_-  ( A[[t]]) =  t^{-1}A[t^{-1}]  \oplus p_- \,  p_+ ( A[[t]]) =
t^{-1}A[t^{-1}]  \oplus g( A[[t]])
\end{multline*}

\end{proof}

For any commutative ring $A$ and
 any element $g \in  \G(A)$  we introduce the  following block matrix  with respect to the decomposition
${A((t)) = t^{-1} A[t^{-1}] \oplus A[[t]]}$ for the action of   $g$  on  $A((t))$:
\begin{equation}
\label{matr'}
\begin{pmatrix}
a_g & b_g  \\
c_g & d_g
\end{pmatrix}
\end{equation}
where $d_g : A[[t]]  \to A[[t]]$, $d_g= \mathop{\rm pr} \cdot (g |_{A[[t]]})$ and $\mathop{\rm pr}: A((t))  \to A[[t]]$
is the projection. (Here  the matrix acts from the left on an element-column from $A((t))$.) We note that the $A$-module map $d_g$ is continuous as the composition of continuous maps.

As an immediate corollary of Proposition~\ref{dir-sum-new} we obtain the statement.

\begin{cons}  \label{block}
 For any commutative ring $A$ and  any $g \in \G^0(A)$ the map $d_g$ is bijective.
\end{cons}

\subsection{Construction of the   determinant central extension of $\G^0$}  \label{constr-det-g0}

We will construct a central extension of the group functor $\G^0$ by the group functor $\GG_m$. We call this central extension as the determinant central extension.  (This construction is similar to  the corresponding constructions for the group $\AutL(A)$ from~\cite[\S~3.3]{O1} and to   the constructions in the smooth case from~\cite[\S~6.6]{PS}.)

Let $A$ be any commutative ring.

 Let $\GL_A(A[[t]])$ be the group of all $A$-module automorphisms of  $A[[t]] $.

We construct a group  $\widehat{\G^0}(A)$ that as a set consists of all pairs $(g, r)$, where $g \in \G^0(A)$ and $r \in \GL_A(A[[t]])$
such that there is an integer $n > 0$, which depends on $g$ and $r$, with the property $d_g |_{t^nA[[t]]} = r |_{t^nA[[t]]}$. We note that
$r$ is  a continuous map, since $d_g -r$ and $d_g$ are continuous maps. Using that $g$, $g^{-1}$ and $r$ are continuous maps, it is easy to see that the set
$\widehat{\G^0}(A)$ is a subgroup in $\G^0(A) \times \GL_A(A[[t]])$. This gives the group structure on the set $\widehat{\G^0}(A)$.

Now we have the homomorphism  $\widehat{\G^0}(A)  \stackrel{\alpha}{\lrto} \G^0(A)$ given as $(g,r)  \mapsto g$. By Corollary~\ref{block}, this homomorphism  is surjective.

Therefore we obtain
an exact sequence of groups
\begin{equation}  \label{GLf}
1 \lrto \GL_f(A)  \lrto \widehat{\G^0}(A)  \stackrel{\alpha}{\lrto} \G^0(A)  \lrto 1  \, \mbox{,}
\end{equation}
where $\GL_f(A)$ is  a group that consists of all elements $r \in \GL_A(A[[t]])$ such that there is an integer $n >0$, which depends on $r$,
with the property $r |_{t^n A[[t]]} = {\rm id}$.

We note that there is a well-defined determinant homomorphism
$$
\det  \, : \, \GL_f(A)  \lrto A^*  \,
$$
where $\det(r)$ for $r \in \GL_f(A)$ is the determinant of the action of $r $ on a free $A$-module  $A[[t]]/ t^n A[[t]]]$, where $r |_{t^n A[[t]]} = {\rm id}$.

\begin{lemma}  \label{lem-det}
For any $r_1 \in \GL_f(A)$  and $(g,r) \in \widehat{\G^0}(A)$ we consider
$$(g^{-1}, r^{-1}) (1,r_1) (g,r)= (1,r_2)  \mbox{,} $$
where $r_2 \in \GL_f(A)$.
Then we have $\det(r_2) = \det(r_1)$.
\end{lemma}
\begin{proof}
We consider an integer $n > 0$ such that $r_1 |_{t^n A[[t]]} = {\rm id}$ and  an integer $l > 0$ such that $r(t^l A[[t]])  \subset t^n A[[t]]$. Then we have $r_2 |_{t^l A[[t]]} = {\rm id}$ and $\det(r_2)$ comes from the following composition of homomorphisms of $A$-modules of rank $1$
$$
\bigwedge^l (A[[t]]/ t^l A[[t]]) \to \bigwedge^l (A[[t]]/ r( t^l A[[t]])) \to \bigwedge^l (A[[t]]/ r( t^l A[[t]]))
\to   \bigwedge^l (A[[t]]/ t^l A[[t]])  \, \mbox{,}
$$
where the first arrow is induced by the map $r$, the second arrow is induced by the map $r_1$, the third arrow is induced by the map $r^{-1}$.
Besides, the second arrow is the multiplication by $\det(r_1)$, since the homomorphism $\det$ does not depend on the choice of a basis of a free $A$-module of finite rank and addition to this $A$-module another projective $A$-module of finite rank, see~\cite[\S~3]{Mi}.
\end{proof}

We denote a subgroup $T(A) = \Ker \det$ of the group $\GL_f(A)$. By Lemma~\ref{lem-det}, $T(A)$ is a normal subgroup in  $\widehat{\G^0}(A)$.
We introduce a group $\widetilde{{\G^0}}(A) = \widehat{\G^0}(A)/ T(A)$. From exact sequence~\eqref{GLf} we obtain a central extension of groups
\begin{equation}  \label{det-g0}
1 \lrto A^*  \lrto \widetilde{{\G^0}}(A)  \lrto \G^0(A)  \lrto 1  \, \mbox{.}
\end{equation}

We call this central extension as {\em the determinant central extension}.

We note that the map $\alpha$ in~\eqref{GLf} has a section $g \mapsto (g, d_g)$ as sets. This induces   a section  $\sigma$ of  the determinant central extension~\eqref{det-g0}. Using the section $\sigma$, in the standard way we obtain a $2$-cocycle
${D  }$ for the determinant central extension
such that ${D(x,y) =  \sigma(x) \sigma(y) \sigma(xy)^{-1} }$, where  $x,y \in \G^0(A)$. From the contruction of $\sigma$ we obtain that the $2$-cocycle $D$
 is explicitly written in the following way
\begin{equation}  \label{coc-D-2}
D(x,y) = \det (d_x \cdot d_y  \cdot d_{xy}^{-1}  )  \, \mbox{,}
\end{equation}
where      $d_{xy}= c_x \cdot b_y + d_x  \cdot d_y$ (recall  formula~\eqref{matr'}). The determinant $\det$ in  formula~\eqref{coc-D-2}  is well-defined, since there is an integer $n \ge 0$ such that
${(d_x \cdot d_y \cdot d_{xy}^{-1})  |_{t^n A[[t]]}  = {\rm id}}$.

Clearly, the constructions of central extension~\eqref{det-g0} and $2$-cocycle~\eqref{coc-D-2} are functorial with respect to the ring~$A$. Therefore we obtain the following proposition.

\begin{prop}  \label{prop-det}
The central extension~\eqref{det-g0} gives the determinant central extension of the group functor $\G^0$ by the commutative group functor $\GG_m$.
This central extension admits a section $\sigma$ which gives a $2$-cocycle $D$  of $\G_0$ by $\GG_m$ given by formula~\eqref{coc-D-2}.
\end{prop}

\subsection{Construction of the   determinant central extension of $\G$}  \label{centr-det-g}

We will construct a central extension of the group functor  $\G$ by the group functor $\GG_m$. We will call this central extension also as the determinant central extension. (In section~\ref{uniqueness}  we will prove that this central extension restricted to $\G^0$  coincides with the determinant central extension from Proposition~\ref{prop-det}.)

Let $A$ be any commutative ring.

By~\cite[\S~3.2]{O1}, for any element  $g$ from the group $\G(A)$  and any integer $l$ such that $t^l A[[t]]  \subset g (A[[t]]) $
 the $A$-module $g(A[[t]]) / t^l A[[t]]$ is projective and finitely generated. For any $g_1, g_2 \in \G(A)$
 the projective $A$-module of rank $1$
 $$
 \Hom\nolimits_A \left(\bigwedge^{\rm max}(g_1(A[[t]])/ t^l A[[t]]),  \, \bigwedge^{\rm max}(g_2(A[[t]])/ t^lA[[t]]) \right)
$$
does not depend on the choice of an integer $l$ up to a unique isomorphism. We identify over all such $l$ all these projective $A$-modules via the following
definition of the projective $A$-module of rank $1$:
$$
\det(g_1(A[[t]]) \mid g_2(A[[t]])) = \varinjlim_l \Hom\nolimits_A \left(\bigwedge^{\rm max}(g_1(A[[t]])/ t^l A[[t]]), \, \bigwedge^{\rm max}(g_2(A[[[t]])/ t^lA[[t]]) \right)  \, \mbox{.}
$$

For any elements $g_1, g_2, g_3 \in \G(A)$ we have a canonical isomorphism of $A$-modules
\begin{equation}
\label{iso-tens}
\det(g_1(A[[t]]) \mid g_2(A[[t]]))  \otimes_A \det(g_2(A[[t]]) \mid g_3(A[[t]])) \xrightarrow{\sim} \det(g_1(A[[t]]) \mid g_3(A[[t]]))
\end{equation}
that satisfies the associativity diagram for any four elements from $\G(A)$.
This gives also the following canonical isomorphisms of $A$-modules
\begin{gather*}
\det(g_1(A[[t]]) \mid g_1(A[[t]])) \xrightarrow{\sim} A  \, \mbox{,}  \\
\det(g_1(A[[t]]) \mid g_2(A[[t]]))  \otimes_A \det(g_2(A[[t]]) \mid g_1(A[[t]])
 \xrightarrow{\sim}
  A  \, \mbox{.}
\end{gather*}
Besides, any element $f \in \G(A)$
  defines an isomorphism of $A$-modules
\begin{equation} \label{mult-isom}
\det(g_1(A[[t]]) \mid g_2(A[[t]])) \xrightarrow{\sim} \det(fg_1(A[[t]]) \mid fg_2(A[[t]]) \mbox{.}
\end{equation}

 \begin{lemma}  \label{det-free}
 For any elements $g_1, g_2 \in \G(A)$ the $A$-module $\det(g_1(A[[t]]) \mid g_2(A[[t]]))$
is free of rank $1$.
 \end{lemma}
 \begin{proof}
 By~\eqref{mult-isom} it is enough to suppose that, for example, $g_1 =1$ and $g_2=g$.
 { If an element $g \in \G^0(A)$, then this statement follows from Corollary~\ref{block} and~\cite[Prop.~3.3]{O1}} (see also below the inverse isomorphism to the isomorphism given by formula~\eqref{isom-det} for $r = d_g$). For any $g \in \G(A)$ we have
 $g = t^{\underline{n}} f$, where $\underline{n} \in \underline{\dz}(A)$ and $f \in \G^0(A)$ (see Lemma~\ref{dec-G}). Therefore the statement follows from  isomorphisms of $A$-modules
 \begin{multline*}
 \det(A[[t]] \mid g(A[[t]])) \simeq \det(A[[t]] \mid t^{\underline{n}}(A[[t]]))  \otimes_A   \det( t^{\underline{n}}(A[[t]])  \mid
 t^{\underline{n}} f (A[[t]])) \simeq  \\ \simeq
 A \otimes_A   \det( t^{\underline{n}}(A[[t]])  \mid
 t^{\underline{n}} f (A[[t]]))
 \simeq
  \det( A[[t]]  \mid
  f (A[[t]]))   \,  \mbox{.}
 \end{multline*}
 \end{proof}

 Using Lemma~\ref{det-free} we define a {\em determinant central extension}
\begin{equation}  \label{det-cent-ext'}
1 \lrto A^*  \lrto \widetilde{\G}(A)  \lrto \G(A)  \lrto  1  \, \mbox{,}
\end{equation}
where the group $\widetilde{\G}(A)$ consists of pairs $(g, s)$, where $g \in \G(A)$
and $s$ is an element of free $A$-module  $\det( g(A[[t]])  \mid  A[[t]])$ of rank $1$ such that
$s$ is a generator of the $A$-module  $\det( g(A[[t]])  \mid  A[[t]])$.
 The group law is as follows
$$
(g_1, s_1) (g_2, s_2)= (g_1 g_2, g_1(s_2) \otimes s_1)  \, \mbox{,}
$$
and we map $(g,s)  \in \widetilde{\G}(A)$ to $g \in \G(A)$ to obtain the central extension.

The correspondence $A \mapsto \widetilde{\G}(A)$ is functorial with respect to $A$  (compare with analogous fact for $\widetilde{\AutL(A)}$  in~\cite[\S~3.3]{O1}). This follows from the natural isomorphism
$$(g (A_1[[t]]) / h (A_1[[t]])) \otimes_{A_1} A_2   \xrightarrow{\sim}   u(g) (A_2[[t]]) / u(h) (A_2[[t]])  \mbox{,}$$
where ${u: A_1 \to A_2}$ is any homomorphism of commutative rings, which induces the homomorphism  ${\G(A_1) \to \G(A_2)}$ denoted by the same letter $u$, and elements $g, h$ are from
$ \G(A_1)$ such that $g (A_1[[t]]) \supset h (A_1[[t]])$. This isomorphism  is satisfied, since  by~\cite[Prop.~3.2]{O1}   the projective $ A_1$-module $g (A_1[[t]]) / h (A_1[[t]]))$ is a direct summand of an $A_1$-module $t^lA_1[[t]]/ t^mA_1[[t]]$ for appropriate $l, m \in \dz$, and for
the $A_1$-module $t^lA_1[[t]]/ t^mA_1[[t]]$ this isomorphism is evident.

Thus we obtain the group functor $\widetilde{\G}$ and  {\em  the determinant central extension} of $\G$ by $\GG_m$ from~\eqref{det-cent-ext'}.

\subsection{Uniqueness of an extension of the determinant central extension from $\G^0$ to $\G$}   \label{uniqueness}
We will prove that the determinant central extension of $\G$ by $\GG_m$  after restriction  to $\G^0$ coincides with the determinant central extension from Proposition~\ref{prop-det} and an extension  from $\G^0$ to $\G$ is unique. The last statement will follow from the statement that any central extension of $\G$ by $\GG_m$ is uniquely defined by its restriction to $\G^0$.

\begin{prop}  \label{mor-gr}
Any morphism of group functors $\G^0_{\Q}  \to {\GG_m}_{\Q}$ (i.~e. which preserves the group structures) is trivial.
Analogously, any morphism of group functors ${\G^0}  \to {\GG_m}$  is trivial.
\end{prop}
\begin{proof}
Let $\beta$ be a morphism of group functors $\G^0  \lrto \GG_m$.
 The functor $\GG_m$ is represented by the affine  scheme, and  the  functor $\G^0$  is represented by an ind-affine ind-scheme that is ind-flat.
  Therefore to prove that $\beta$ is trivial,
  it is enough to prove that the morphism $\beta_{\Q} : \G^0_{\Q}  \lrto {\GG_m}_{\Q}$ (i.~e. after restrictions of functors to the category of  commutative $\Q$-algebras) is trivial, since these morphisms are determined by the corresponding homomorphisms of rings of regular functions of schemes and ind-schemes.

  So, we will prove the first statement of the proposition and suppose that we have a morphism of group functors $\beta : \G^0_{\Q}  \to {\GG_m}_{\Q} $.

Let $A$ be any commutative  $\Q$-algebra.

We claim that for any $z  \in \G^0(A)$  the differential $(d \beta)_z$ of $\beta$ is zero as the map from the tangent $A$-module $T_z \G^0 (A)$ of $\G^0$ at $z$ to the tangent $A$-module $T_{\beta(z)} {\GG_m}(A)$, where
${
T_z \G^0 (A) = \rho^{-1}(z)
}$
and ${\rho \, :  \, \G^0(A[\ve]/ (\ve^2))  \to G^0(A) }$ is the natural map, and  the tangent space
$$T_{\beta(z)} {\GG_m}(A)= \beta(z) + A \ve  \subset (A[\ve] / (\ve^2))^*$$
is naturally identified with $A$.

Indeed, if $z$ is the identity element $e$ of the group $\G^0(A)$, then $T_e \G^0 (A) = \Lie \G^0 (A)$ and $(d \beta)_e$ is the homomorphism from
the Lie $A$-algebra $\Lie \G^0 (A)$ to the Lie algebra $\Lie \GG_m(A)$.  But by Lemma~\ref{commutant}, we have ${\left[\Lie \G^0 (A), \Lie \G^0 (A) \right] =
\Lie \G^0 (A)}$,   and $\Lie \GG_m(A)  \simeq A$ is an Abelian Lie algebra. Therefore  $(d \beta)_e = 0$.

For any $z \in \G^0(A)$ we have $(d \beta)_z = 0$, since $\beta$ is the morphism of group functors and
 therefore the following diagram is commutative
 $$
 {\xymatrix{
   {T_e \G^0(A)} \ar[d]_{(d \beta)_e}  \ar[rr]^{(d r_z)_e} &&  {T_z \G^0(A)} \ar[d]^{(d \beta)_z}  \\
   {T_{1} {\GG_m}(A)} \ar[rr]_{(d r_{\beta(z)})_1}   && {T_{\beta(z)} {\GG_m}(A)}
 }}
 $$
 where $r_z$ or $r_{\beta(z)}$ is the multiplication on the right by the element $z$ or $\beta(z)$,
 and the horizontal arrows are isomorphisms.

 According to Section~\ref{ex-ind-sch}, the functor $\G^0_{\Q}$ is represented by an ind-scheme  \linebreak  $\mbox{``$\varinjlim\limits_{i \in I}$''} \Spec C_i$ for concrete $\Q$-algebras $C_i$. Therefore the morphism $\beta$ is uniquely defined by the homomorphism of $\Q$-algebras
 $$\beta^*   \; : \;   \Q[x,x^{-1}]  \lrto \oo(\G^0_{\Q}) =  \varprojlim\limits_{i \in I} C_i   \, \mbox{,}$$
which is uniquely defined by the image $F = \beta^* (x)$ in $\oo(\G^0_{\Q})$.

By~\eqref{decg0},  \eqref{declg}, \eqref{decaut}, this ind-scheme is the product of another ind-schemes written explicitly as some ind-schemes
from~\eqref{ind-formb}-\eqref{ind-formc} after extension of scalars to $\Q$. Thus, the $\Q$-algebra $\oo(\G^0_{\Q})$
is the algebra of formal power series in infinite (countable) number of variables with coefficients in the $\Q$-algebra of polynomials in infinite (countable) number of variables with coefficients in the $\Q$-algebra of Laurent polynomials in two variables. A formal power series from this algebra is an infinite sum of monomials, where every monomial is the product of a finite number of variables in powers, and the convergence  condition for this infinite sum  comes from the description of ideals $I_{\{ \epsilon_i \}}$ in Section~\ref{ex-ind-sch}.

We recall that  for any commutative $\Q$-algebra $A$ and  any $z \in \G^0(A)$ we have ${(d \beta)_z = 0}$. Therefore  from explicit description of the tangent $A$-module $T_z \G^0(A)$ as an inductive limit of  $A$-modules $A^{\mathbb{N}}$ and
by the Taylor formula (which is enough to see for any monomial)
$$
F (z +   \ve_c) = F(z) + \left(\frac{\partial F}{ \partial c}  \right) (z) \ve_c  \, \mbox{,} \qquad \ve_c^2 =0 \, \mbox{,} \qquad \ve_c  \quad \mbox{is placed in $c$th coordinate, }
$$
we obtain that the series $\frac{\partial F}{ \partial c}$ obtained as the partial derivative of the series $F$ with respect to any variable $c$ (including variables from the algebra of polynomials and the algebra of Laurent polynomials) applied to the element $z$ equals zero.

But if a series applied to any element from $\G^0(A)$ for any commutative $\Q$-algebra $A$ equals zero, then this series is zero itself. Indeed,
it is enough to take algebras of the following forms
$$A = \Q[x_1, x_2, x_1^{-1}, x_2^{-1}]  \otimes_{\Q} \Q[y_1, y_2, \ldots ] \otimes_{\Q} \left( \Q[z_1, \ldots, z_l]/ (z_1^{n_1}, \ldots, z_l^{n_l})
\right) $$ and put $x_i$, $y_j$, $z_k$  and zero instead the variables in the series $F$ (including variables from the algebra of polynomials and the algebra of Laurent polynomials).

Thus, we obtained that $\frac{\partial F}{ \partial c} =0$ with respect to any variable $c$.
Hence $F = {\rm const}  \in \Q $.

Since, the morphism $\beta$ preserves the group structure, this constant equals $1$, i.e. $F =1$ and $\beta$ is trivial.

\end{proof}

\begin{Th}  \label{th1}
Any central extension of the group functor $\G$ by the group functor $\GG_m$
is uniquely (up to isomorphism) defined by its restriction to the group subfunctor $\G^0$.
The same is also true after restrictions of all these functors to the category of commutative $\Q$-algebras.
\end{Th}
\begin{proof}
By Lemma~\ref{dec-G}, we have $\G = \G^0  \rtimes \underline{\dz}$.

Similarly to~\cite[1.7. Construction]{BD},   a central extension $\tilde{C}$ of a group functor \linebreak ${C = G \rtimes H}$
by  a group functor $E$ is equivalent to the following data:
\begin{itemize}
\item[1)]   a central extension of $H$ by $E$;
\item[2)]   a central extension $\tilde{G}$ of $G$ by $E$;
\item[3)] an action of $H$ on the group functor $\tilde{G}$, lifting the action of $H$ on $G$ and  trivial on $E$.
\end{itemize}

In our case any central extension  $\tilde{\underline{\dz}}$ of $\underline{\dz}$ is trivial, since there is a group functor splitting
$\underline{\dz} \to \tilde{\underline{\dz}}$.
This splitting is given by any map
\begin{equation} \label{splitt}
\underline{\dz}  (\dz) = \dz \ni 1  \longmapsto a \in\tilde{\underline{\dz}}(\dz)
\end{equation}
that maps $1$ to a preimage $a$ of $1$ in  $\tilde{\underline{\dz}}(\dz)$.
Now for any commutative ring $A$ and any function $\underline{n} \in \underline{\dz}(A)$ we have $A = A_1 \times \ldots \times A_l$ such that
the function $\underline{n}$ restricted to any $\Spec A_i$ is the constant function with value $n_i \in \dz$.
Besides, we have ${\tilde{\underline{\dz}}(A) = \tilde{\underline{\dz}}(A_1) \times \ldots \times \tilde{\underline{\dz}}(A_l)}$.
Now $n_i \in \underline{\dz}(A_i)$
is the image of $n_i \in \underline{\dz}  (\dz) = \dz$. Therefore  the splitting
$$ \underline{\dz}(A_i) \ni n_i \longmapsto a^{n_i}  \in \tilde{\underline{\dz}}(A_i) $$
follows from
formula~\eqref{splitt}.

We claim uniqueness of the   action of $1 \in \underline{\dz}  (\dz)$ on  the group functor that gives a central extension
of $\G^0$ by $\GG_m$ such that this action  lifts the action of $1 \in \underline{\dz}  (\dz)$ on $\G^0$ and is  trivial on $\GG_m$. Indeed,
if there are at least two such actions, then they differ by the automorphism of the central extension (this automorphism induces identically action on $\GG_m$ and $\G^0$). But such  automorphism can be identified with the morphism of group functor $\G^0 \to \GG_m$. By Proposition~\ref{mor-gr}
this morphism is trivial.

Now for any commutative ring $A$, by similar reasonings as above,  the action of ${1 \in \underline{\dz}  (\dz)}$ induces the action of $\underline{n} \in \underline{\dz}(A)$ on the group which defines the central extension after specifying the ring $A$. Therefore the action of the group functor $\underline{\dz}$ is unique.

The statement about the restricted to commutative  $\Q$-algebras functors is proved by the same reasonings with the help of Proposition~\ref{mor-gr}.
\end{proof}

\begin{Th}
The determinant central extension from Section~\ref{centr-det-g}
\begin{equation}  \label{det-big}
1 \lrto \GG_m  \lrto \widetilde{\G}  \lrto \G \lrto 1
\end{equation}
is a unique (up to isomorphism) central extension of $\G$ by $\GG_m$ such that its restriction to $\G^0$ coincides with the central extension from Proposition~\ref{prop-det}. (The same is also true after restriction to the category of commutative $\Q$-algebras.)
\end{Th}
\begin{proof}
By Theorem~\ref{th1} it is enough to prove that the determinant central extension from Section~\ref{centr-det-g}
restricted to  $\G^0$ is isomorphic to the central extension from Proposition~\ref{prop-det}.

Let $A$ be a commutative ring. We construct a homomorphism from exact sequence of groups~\eqref{GLf}
to the central extension~\eqref{det-cent-ext'}  restricted to $\G^0(A)$.

Let $(g, r)  \in \widehat{\G^0}(A)$. Then there is an $A$-submodule $L =t^m A[[t]]  $ with $m \ge 0$
such that  $g(L)  \subset A[[t]]$ and  $r |_L = d_g |_L$  (we recall notation from~\eqref{matr'}).
Since $g(L) \subset A[[t]]$, we have $d_g |_L = g |_L$.  Therefore $r(L)= g(L)$. Then we have the following isomorphism
\begin{equation}  \label{isom-det}
\frac{A[[t]]}{g(L)}= \frac{A[[t]]}{r(L)}  \stackrel{r}{\longleftarrow}  \frac{A[[t]]}{L}
\stackrel{g^{-1}}{\longleftarrow} \frac{g(A[[t]])}{g(L)}
\end{equation}

By~\cite[\S~3.2]{O1}, all the above $A$-modules are projective. Therefore taking some ${N =t^k A[[t]] \subset g(L)}$, we obtain the isomorphism of $A$-modules
$g(A[[t]])/ N \to A[[t]]/N$, which gives the isomorphism of the corresponding top exterior powers of these projective $A$-modules. This gives an element  $ s $ from the $A$-module $\det(g (A[[t]] | A[[t]])$, and this element generates this $A$-module.
 Thus we have mapped the element
$(g, r)  \in  \widehat{\G^0}(A)$ to   the element $ (g,s) \in \widetilde{\G}(A)$. This map  is a group homomorphism. Besides, this map restricted to the subgroup  $\GL_f(A)$ is the determinant homomorphism $\det$.

\end{proof}

\begin{nt} \em
From the proof of Theorem~\ref{th1} it follows that to extend the determinant central extension from $\G^0$ to $\G$ it is enough  to  lift the action of $1 \in \underline{\dz}  (\dz)$ on $\G^0$
to $\widetilde{{\G^0}}$
such that it will be  trivial on $\GG_m$. It is possible also to introduce this extension through the groups
$\GL_f(A)$  and $\widehat{\G^0}(A)$ introduced in Section~\ref{constr-det-g0} (this is similar to the smooth case from~\cite[\S~6.6]{PS}).
Indeed, let $A$ be any commutative ring.  The action of $1 \in \underline{\dz}  (\dz)$ on $A((t))$
is the multiplication by the element $t$, and the action on $\G^0 (A)$ is by means of the conjugation by $t$.  Now we consider the following map on $\widehat{\G^0}(A)$:  $(g,r)  \mapsto (t g t^{-1}, r_t )$, where $r_t \, :  \, A[[t]]  \to A[[t]]$ equals to ${\rm id}  \oplus t r t^{-1} $ for the decomposition of $A$-module $A[[t]]= A \oplus t A[[t]]$. This map is an injective endomorphism of the group $\widehat{G^0}(A)$, but this endomorphism is not surjective. But it is easy to see that the induced endomorphism of the group $ \widetilde{{\G^0}}(A)$ is the automorphism. This automorphism is the lift of the action of  $1 \in \underline{\dz}  (\dz)$ to the group $\widetilde{{\G^0}}(A)$.
\end{nt}

\begin{nt} \label{can-sec} \em
We note that determinant central extension~\eqref{det-big} has a natural section $ { \vsgm : \G  \to \widetilde{\G}}$ as functors (not as group functors),
since $\G = \G^0  \rtimes \underline{\dz}$, and the central extension has a canonical section $\sigma$ over $\G^0$ by Proposition~\ref{prop-det}  and a group section $\varrho$ over $\underline{\dz}$ as
$$t \longmapsto (t,1)  \, \mbox{,} \qquad \mbox{where} \quad 1 \in \det (tA[[t]] \mid  A[[t]]) \simeq A[[t]] /  tA[[t]] \simeq A  \, \mbox{.}$$
Now for any $g \in \G^0(A)$, $\underline{n} \in \underline{\dz}(A)$ we put
$$
\vsgm ((g, t^{\underline{n}} )) = \sigma(g) \varrho(t^{\underline{n}})  \, \mbox{.}
$$
The corresponding $2$-cocycle on $\G$ with coefficients in $\GG_m$ coincides with the  {$2$-cocycle} $D$ after restriction  to  $\G^0 $.
We denote the $2$-cocycle on $\G$  by the same letter~$D$.
\end{nt}

By Remark~\ref{can-sec}  we have a non-group section of determinant central extension. Therefore   we have a decomposition of  functors (not as group functors)
$$
\widetilde{\G} \simeq  \GG_m \times \G \, \mbox{.}
$$

Since functors $\GG_m$ and $\G$ are represented by ind-affine ind-schemes, we obtain that functor $\widetilde{\G}$ is also represented by an ind-affine ind-scheme.

\section{Geometric action of $\G$ and of $\widetilde{\G}$}

\label{geom-act}

We will construct a natural  action of $\G$ on the moduli stack   of some geometric data (including proper families of  curves with some conditions and invertible sheaves on these families) and a natural action of $\widetilde{\G}$ on the   determinant line bundle on this moduli stack.

During this section $A$ is any commutative ring, if we do not specify the concrete ring.

\subsection{Quintets and proper quintets}   \label{quint}

Let $C$ be a  separated family of curves over $A$, i.e. we have a  separated morphism  ${C \to \Spec A}$ whose fibres are one-dimensional schemes. We consider  an $A$-point ${p \in C(A)}$
such that $C$ is smooth near $p$, i.e. there is an open $V \supset p$ such that $V$ is smooth over $\Spec A$ (we will denote by the same letter an $A$-point and its image in $C$).  Since the morphism ${C \to \Spec A}$ is separated,
$p$ is a closed subscheme of $C$. Besides,
by~\cite[Theorem 17.12.1]{EGA4}   $p$  is an effective  Cartier divisor on $C$.

\begin{nt}  \label{Cartier_smooth} \em
 There is also a statement in the opposite direction.  By~\cite[Theorem 17.12.1]{EGA4} if an $A$-point $q \in C(A)$ defines the Cartier divisor on $C$ and the morphism ${C \to \Spec A}$ is locally of finite presentation near $q$, then the morphism ${C \to \Spec A}$ is smooth near~$q$.
\end{nt}

We consider the $A$-algebra
$$
\hat{\oo}_{C, p} = \varprojlim_{n > 0}  H^0(C, \oo_C / \oo_C( - np))   \, \mbox{,}
$$
which is the $A$-algebra of functions on the formal neighbourhood of $C$ at $p$. It is clear that $\hat{\oo}_{C, p}$ has a natural topology that makes  $\hat{\oo}_{C, p}$ into a topological $A$-algebra with the discrete topology on $A$.   We consider also the topological $A$-algebra
$$
{\mathcal K}_{C,p} =   \varinjlim_{m  > 0}  \varprojlim_{n > 0}  H^0(C, \oo_C (mp) / \oo_C( - np))  \,  \mbox{,}
$$
which is the $A$-algebra of functions on the punctured formal neighbourhood of $C$ at $p$.

Let $\ff$ be a  sheaf of $\oo_C$-modules  such that $\ff$ is an invertible sheaf near $p$. We consider the $\hat{\oo}_{C,p}$-module
$$
\hat{\ff}_{C, p} = \varprojlim_{n > 0}  H^0(C, \ff / \ff( - np)) \, \mbox{,}
$$
which is the module of sections of $\ff$ over the formal neighbourhood of $C$ at $p$.

\begin{defin}  \label{def-quint}
By a quintet over $A$ we call a collection $(C, p, \ff, t, e)$, where
\begin{itemize}
\item $C$ is a  separated family of curves over $A$,
\item
 $p \in C(A)$ is an $A$-point of $C$
such that $C$ is smooth near $p$,
\item$\ff$ a  sheaf of $\oo_C$-modules such that $\ff$  is an invertible sheaf near $p$,
\item $t$ is a relative formal parameter at $p$, i.e.
$t$ is an element of ideal of $p$ of $\hat{\oo}_{C, p}$ that induces an isomorphism of topological $A$-algebras
\begin{equation}  \label{complt}
A [[t]]     \xrightarrow{\sim}  \hat{\oo}_{C, p}  \, \mbox{,}
\end{equation}
\item $e$ is a formal trivialization of $\ff$ at $p$, i.e.  $\hat{\ff}_{C, p}$ is a  free   $\hat{\oo}_{C, p}$-module of rank $1$, and  $e$ is its  basis.
\end{itemize}

A quintet  $(C, p, \ff, t, e)$ is called a proper quintet if, in addition, the natural morphism $C  \to \Spec A$ is a flat proper finitely presented morphism such that all geometric fibers are integral one-dimensional schemes, and  $\ff$ is an invertible sheaf of  $\oo_C$-modules.
\end{defin}

\begin{nt} \em
Since $p$ is a Cartier divisor on $C$, isomorphism~\eqref{complt} exists locally on $\Spec A$ (see~\cite[Corollaire~16.9.9]{EGA4}, where there is also a statement in the opposite direction). In the definition of a quintet (and of a proper quintet) we demand that this isomorphism exists globally on $\Spec A$ and fix a relative formal parameter $t$. Analogously,  a formal trivialization $e$ exists always locally on $\Spec A$, but we demand that it exists globally on $\Spec A$ and fix it.
\end{nt}

We note that for a quintet  $(C, p, \ff, t, e)$ the relative formal parameter $t$ defines an isomorphism
\begin{equation}  \label{K-isom}
A ((t))     \xrightarrow{\sim}  {\mathcal K }_{C, p}  \, \mbox{.}
\end{equation}

\bigskip

Clearly, when we vary $A$ we obtain the moduli stack (as an abstract stack in Zariski topology) of quintets
$\mm$ and, similarly, the moduli stack  of proper quintets $\mm_{\rm pr}$.
Using descent data and morphisms of descent data we can consider $\mm$ and $\mm_{\rm pr}$ as stacks over the category of all schemes (these are stacks associated with the prestacks that are given as stacks over affine schemes and empty over other schemes).

We consider also the corresponding  functors $\qq$ and $\qq_{\rm pr}$ of isomorphism classes of objects of these stacks as the covariant functors from the category of commutative rings to the category of sets:
\begin{gather*}
{ \mathcal{Q}} (A) = \{ \mbox{quintets over $A$} \}/ \, \mbox{isomorphisms}   \\
{ \mathcal{Q}}_{\rm pr}(A) = \{ \mbox{proper quintets over $A$} \}/ \, \mbox{isomorphisms}   \, \mbox{.}
\end{gather*}

There is a natural notion of an action of a group functor from the category of schemes (which will be in our case the sheaf of groups that is represented by an ind-scheme) on a
category fibered in groupoids, see, e.g.,~\cite{R}. In particularly, we can consider  a stack  as an example of such category.
Clearly, this action induces the corresponding action of the group functor on the functor of  isomorphism classes of objects of this stack.

\begin{Th}  \label{action}
\begin{enumerate}
\item
There is a natural action of the group ind-scheme $\G$ on the moduli stack  $\mm$ and on the moduli stack $\mm_{\rm pr}$. This action applied to a quintet
$(C, p, \ff, t, e)$ does not change the geometric fibers of $C \to \Spec A$ and does not change the topological space of $C$.
\item
Fixing a field $k$, the corresponding action of the Lie algebra $\Lie \G(k)$ is transitive after restriction to a proper quintet
 $(C, p, \ff, t, e)$
 with a smooth curve $C$ over $k$. In other words, $\Lie \G(k)$ is mapped surjectively to the tangent space of $\qq_{\rm pr}$ at  $(C, p, \ff, t, e)$  under the action on the quintet obtained from $(C, p, \ff, t, e)$  as the  base change  from $k$ to $k[\ve]/ (\ve^2)$. (The tangent space  consists of all elements from $\qq_{\rm pr}( k[\ve]/(\ve^2))$ giving  $(C, p, \ff, t, e)$ after reduction to $k$.)
\end{enumerate}
\end{Th}
\begin{nt} \em
An action of $\G$ on $\mm$, which will constructed in the proof below,  depends on the choice  of an affine open covering of a neighbourhood of $p$ in $C$ for every quintet $(C, p, \ff, t, e)$ over $A$. The
result of the action of $\G(A)$ on $(C, p, \ff, t, e)$  depends on this covering up to a canonical isomorphism. Therefore the
 corresponding action of $\G$ on $\mathcal Q$ does not depend on these choices.
\end{nt}
\begin{proof}

{\em 1. We prove item~1 of Theorem~\ref{action}.}

Given a quintet $(C, p, \ff, t, e)$ over $A$ and an element $(h, \vp ) \in \G(A)$ we will define a new quintet $(C'', p'', \ff'', t'', e'')$ over $A$.
First we define a quintet $(C', p', \ff', t', e')$ over $A$ from the quintet $(C, p, \ff, t, e)$ and the element $\vp \in \AutL(A)$  (see~\cite[\S~17.3]{FB}, \cite[\S~4.1]{Pol}, but we provide detailed proofs below). Then we will define a quintet
$(C'', p'', \ff'', t'', e'')$ from the quintet $(C', p', \ff', t', e')$ and the element $h \in L \GG_m(A)$  (cf.~\cite[\S~18.1.3]{FB}). The procedure is as follows.

We can suppose that there is an open affine neighbourhood $U$ of $p$ in $C$ such that $U$ is smooth over $\Spec A$,
 the ideal of effective Cartier divisor $p$ restricted to $\Spec \oo(U)$  is generated  by nonzerodivisor of  $\oo(U)$,
 and  $\ff \mid_U \simeq \oo_U$. Indeed, we can do it by taking smaller  affine open subset in $  \Spec A$, and then to obtain new quintets over this new base. Then we can glue the resulting quintets to obtain a new quintet over $\Spec A$.

From the properties of $U$ ut follows that the natural map $\oo(U) \to  {\oo(U \setminus p)}$ is an embedding. Besides, since $U$ is smooth over $\Spec A$, we have that the homomorphism  $A \to \oo(U)$ is smooth, and hence this homomorphism is of finite presentation and flat, see, e.g., \cite[Tag 01V4, Tag 01TO]{Stacks}.

\medskip

{\em Construction of $(C', p', \ff', t', e')$.}

We have the natural homomorphisms of rings
$$
\oo(U) \lrto \hat{\oo}_{C, p}   \mbox{,}  \qquad     \gamma \, :  \, \oo(U \setminus p)  \lrto  {\mathcal K }_{C, p}  \mbox{,}
$$
which lead to the following Cartesian diagram
$$
 {\xymatrix{
   {\oo(U)} \, \ar[d] \ar@{^{(}->}[r] &  {\oo(U \setminus p)} \ar[d]_{\gamma}  \\
   { \hat{\oo}_{C, p}} \, \ar@{^{(}->}[r]   & {{\mathcal K }_{C, p}}
 }}
 $$
Using the fixed relative formal parameter $t$, we can replace the bottom row of this diagram by the embedding $A[[t]]  \hookrightarrow A((t))$.

Now we define $C'$ by gluing the new  affine curve $U'$   with $C \setminus p$ along $U \setminus p$, where the ring $\oo(U')$ is defined as the fibered product (more on existing and properties of such affine curve $U'$ we will write a little bit later)
\begin{equation}  \label{cartesian}
 {\xymatrix{
   {\oo(U')} \, \ar[d] \ar@{^{(}->}[r] &  {\oo(U \setminus p)} \ar[d]_{\vp \, \gamma}  \\
   { A[[t]] \, \ar@{^{(}->}[r]}   & {A((t))}
 }}
 \end{equation}
Here $\vp  \, \gamma$ is the composition of the map $\gamma$ and the map $\vp : A((t))  \to A((t))$.
The composition of homomorphisms $\oo(U') \to A[[t]] \to A$ defines an $A$-point ${p' \in U'(A) \subset C'(A)}$. We will prove a little bit later that $p'$ defines an effective Cartier divisor on $U'$ such that  $U' \setminus p'  \simeq U \setminus p$, and hence $C' \setminus p' \simeq C \setminus p$.

 From the properties of the construction, which we describe below, it will follow
 that the homomoprhism $\oo(U') \to A[[t]]$ (written as left vertical arrow in~\eqref{cartesian}) gives the isomorphism of the completion of $\oo(U')$ with respect to the ideal of $p'$ with $A[[t]]$. Via this isomorphism we define $t'$ as the element that is mapped to $t$.

  We define $\ff'$ and $e'$
such that $\ff' \mid_{C' \setminus p'} = \ff  \mid_{C \setminus p}$ and
using that $\ff \mid_U \simeq \oo_U$, and so that $\ff'$ preserves this property, i.e. $\ff' \mid_{U' \setminus p'} \simeq \oo_{U' \setminus p'}$ and we extend this sheaf trivially to $U'$ by means of this isomorphism.  We put $e'=e$.

\medskip

{\em First properties of the construction.}

Now we say more on the ring $\oo(U')$.

We note that if  $\psi \in \AutL(A)$ and  $\psi (t A[[t]]) = t A[[t]]$, then Cartesian diagram~\eqref{cartesian}
can be changed to the following Cartesian diagram
\begin{equation}  \label{cartesnew}
{\xymatrix{
   {\oo(U')} \, \ar[d] \ar@{^{(}->}[r] &  {\oo(U \setminus p)} \ar[d]_{  \psi\vp \, \gamma}  \\
   { A[[t]] \, \ar@{^{(}->}[r]^{\psi}}   & {A((t))}
 }}
\end{equation}
and we can consider the ring $A((\psi(t))) \simeq A((t))$ with the changed relative formal parameter from $t$ to $\psi(t)$.
Therefore by this procedure,  to investigate $\oo(U')$ and $C'$ we can and will suppose further that $t \in \oo(U)$ such that $\oo(U)/t \oo(U) \simeq A$ (see above our assumptions on $U$). From~\eqref{complt}  we have also that
\begin{equation}  \label{e1}
A[t, t^{-1}]  \subset \oo(U \setminus p)  \, \mbox{.}
\end{equation}

By~\eqref{decaut}, we have $\vp = \vp_+ \vp_-$, where $\vp_+ \in \AutLp(A)$  and $\vp_- \in \AutLm(A)$.  Since $\vp_+ (A[[t]])= A[[t]]$, the ring $\oo(U')$ depends only on the element $\vp_-$, which is characterized by the property~\eqref{ami} that implies
$$\vp_-(t) \in  tA[t^{-1}] \subset A[t,t^{-1}] \subset  \oo(U \setminus p)  \mbox{.}$$
Therefore, using~\eqref{cartesnew},  we  can and will  suppose further that $\vp  \in  \AutLm(A)$, and therefore  $\vp(t) \in \oo(U \setminus p)$. It implies also that ${\vp^{-1}  \in  \AutLm(A)}$,  and therefore by~\eqref{ami}  we have
$$\vp^{-1}(t) \in tA[t^{-1}] \subset A[t,t^{-1}] \subset  \oo(U \setminus p)  \mbox{.}$$

\smallskip

Definition of $\oo(U')$ from the diagram~\eqref{cartesian} can be written as the following
\begin{equation}  \label{detai}
\oo(U') = \left\{ w \mid w \in \oo(U \setminus p) \quad \mbox{such that} \quad \vp \gamma (w) \in A[[t]]     \right\}  \mbox{.}
\end{equation}

Since $\vp$ is a continuous automorphism of $A((t))$, there is an integer $l > 0$ such that $\vp (t^l A[[t]])  \subset A[[t]]$.
We have $\gamma(t^l \oo(U))  \subset t^l A[[t]]$. Therefore from~\eqref{detai} we obtain that
\begin{equation}  \label{e2}
\vp \gamma(t^l \oo(U))  \subset  A[[t]]  \qquad \mbox{and}  \qquad
t^l \oo(U)  \subset \oo(U')  \, \mbox{.}
\end{equation}

From our choice of $t$ we have a canonical isomorphism of $A$-modules
$$ A[t, t^{-1}]/ t^l A[t] \simeq \oo(U \setminus p) / t^l \oo(U)  \mbox{.}$$
Hence and from ~\eqref{e1}  we obtain that the $A$-module
$\oo(U \setminus p)$ is generated   by the following two  $A$-submodules:
\begin{equation} \label{foro}
\oo(U \setminus p) =
A[t, t^{-1}] +
t^l \oo(U)  \, \mbox{.}
\end{equation}

We denote
an element
\begin{equation}  \label{eq-b}
b = \vp^{-1}(t)  \in tA[t^{-1}] \subset A[t,t^{-1}] \subset \oo(U \setminus p) \, \mbox{.}
\end{equation}
From~\eqref{detai} we have that $b \in \oo(U')$.

Since $\vp$ is an automorphism of  $A[t,t^{-1}]$, from~\eqref{detai}, \eqref{foro} and \eqref{e2}
it is easy to see that the $A$-submodule $\oo(U')$ of $\oo(U \setminus p)$
is generated by the following two $A$-submodules:
\begin{equation}  \label{dec-mod}
\oo(U') = A[b] + t^l \oo(U)  \, \mbox{.}
\end{equation}

Hence we immediately obtain that the subring $\oo(U') $ of ring $\oo(U \setminus p)$  is generated  by the subrings $A[b]$ and $A + t^l \oo(U)$.

\medskip

{\em Flattness property and  base change.}

We claim that $A$-module
$\oo(U')$ given by~\eqref{dec-mod} is decomposed into the following  direct sum of $A$-submodules:
\begin{equation}  \label{dec-mod2}
\oo(U')= \bigoplus_{i =0}^{l-1} A b^i \bigoplus t^l \oo(U) \, \mbox{.}
\end{equation}
Indeed, we have $\vp(\oplus_{i =0}^{l-1} A b^i) = \oplus_{i =0}^{l-1} A t^i  $. Besides, we have that $(\oplus_{i =0}^{l-1} A b^i) \cap  t^l \oo(U) = 0$, since $b = \vp^{-1}(t)  \in tA[t^{-1}] $ and therefore $(\oplus_{i =0}^{l-1} A b^i) \subset t^{l-1} A[t^{-1}]$, and we use that $t^{l-1} A[t^{-1}]  \cap t^l \oo(U) =0$. Now \eqref{dec-mod} implies that   \eqref{dec-mod2} will follow now if we will prove that any element $b^m$, where $m \ge l$,
belongs to the $A$-module in the right hand side of~\eqref{dec-mod2}. We will prove it by induction on $m$, where the base of induction is the evident case $m=l-1$. Suppose we proved it for some $m \ge l-1$ and all natural numbers less than $m$. We have $t^{m+1} \in t^l \oo(U) \subset \oo(U')$. Therefore using~\eqref{detai} and since $\vp(t) $ is of type~\eqref{ami}, the element  $s=\vp(t^{m+1}) $ is a polynomial from $A[t]$ whose degree is $m+1$ and the coefficient at $t^{m+1}$ is $1$.
 We apply $\vp^{-1}$ to the polynomial $s$, and using that $b= \vp^{-1}(t)$, we obtain that $t^{m+1} = \vp^{-1}(\vp(t^{m+1})) = \vp^{-1}(s)$
can be written as the element from $A[b]$, where this polynomial in variable $b$ has degree $m+1$  and the coefficient at $b^{m+1}$ is $1$. Using that $t^{m+1} \in t^l \oo(U)$ and the induction hypothesis, we obtain that $b^{m+1}$ belongs  to the $A$-module in the right hand side of~\eqref{dec-mod2}.

Since $\oo(U)$ is a flat $A$-module, we obtain from~\eqref{dec-mod2} that $\oo(U')$ is also a flat $A$-module.

\smallskip

From description of the construction $\oo(U')  \subset \oo(U \setminus p)$ by formula~\eqref{dec-mod2} it follows that this construction commutes with base change from $A$ to another commutative ring.

\medskip

{\em Cartier divisor $p'$, completion at $p'$ and complement to $p'$.}

We have that the $A$-algebra $\vp(\oo(U'))$ is an $A$-subalgebra of $A[[t]]$. The point  ${p' \in U'(A)}$ is defined through  the following  homomorphism
$$ \oo(U') \stackrel{\vp}{\lrto} \vp(\oo(U')) \stackrel{\sigma}{\lrto} A  \, \mbox{,}$$ where $\sigma$ is induced  by the natural homomorphism $A[[t]] \to A$. We claim that that ideal $ \Ker \sigma $ that is the kernel of $\sigma$ is  generated by the element $t = \vp(b) \in \vp(\oo(U')) $. Indeed, we suppose that $v \in \oo(U')$ such that $\vp(v)  \in tA[[t]]$. We have $b^{-1} = (\vp^{-1}(t))^{-1} \in t^{-1}A[t^{-1}]$, since $b \in tA[t^{-1}] $ is of type~\eqref{ami}. Therefore $b^{-1} \in \oo(U \setminus p)$. Hence $\vp(v) = (\vp(v) t^{-1}) t$, where $\vp(v)t^{-1} = \vp(vb^{-1}) \in \vp(\oo(U'))$, because $vb^{-1}  \in \oo(U \setminus p)$ and  $\vp(vb^{-1}) \in A[[t]]$. Thus we proved that $ \Ker \sigma  = t \vp(\oo(U')) $, and hence $p'$ is an effective Cartier divisor on~$U'$.

Besides, from these reasonings it follows that the completion of $\vp(\oo(U'))$ with respect to the ideal $t \vp(\oo(U'))$ is $A[[t]]$.

We have that $\oo(U') \subset \oo(U \setminus p)$. We claim that $\oo(U' \setminus p') = \oo(U \setminus p)$. Indeed, we have proved that
 $\oo(U' \setminus p') = \oo(U')[b^{-1}]$. Since for any integer $m$ we have $b^m \in t^mA[t^{-1}]$, we obtain $\oo(U' \setminus p') \subset \oo(U \setminus p)$. To obtain the reverse inclusion, we note that $t^{-1} \in A[b, b^{-1}]$, since $\vp$ is the automorphism of $A[t, t^{-1}]$ and hence $A[t,t^{-1}] = A[b, b^{-1}]$.

\medskip

{\em Finite presentability and smoothness property.}

 Now we prove that the ring homomorphism  $A \to \oo(U')$ is of finite presentation. We know that $A \to \oo(U)$ is of finite presentation. Hence
 there is a surjective $A$-algebra homomorphism
  $\kappa \, : \, A[y_1,  \ldots, y_r]  \to \oo(U) $  such that $\kappa(y_i) \in t\oo(U)$ for any $1 \le i \le r$. Let $\omega_1, \ldots, \omega_m$ be all monomials  of elements $y_1, \ldots, y_r$ of total degrees from $l$ to $2l-1$.
 Since $A[t]/ t^l A[t] \simeq \oo(U)/ t^l \oo(U)$, we obtain a surjective $A$-algebra homomorphism
  $$\theta' \, : \, A[x_0,  \ldots, x_m]  \to \oo(U) \, \mbox{,}$$
   where $\theta'(x_0) =t $, $\theta'(x_i) = \kappa(\omega_i)$  for $1 \le i \le m$. We note that   $\theta'(T) =t^l \oo(U)$, where $T$ is the $A$-submodule generated by all monomials of elements $x_1, \ldots, x_m$ of total positive degrees.

Let $ \theta \, : \,   A[x_0, x_0^{-1}, x_1,  \ldots, x_m]  \to \oo(U \setminus p)$ be the natural surjective $A$-algebra homomorphism that extends $\theta'$.
 Since $A \to \oo(U)$ is of finite presentation, the ideal $\Ker \theta'$ is finitely generated by polynomials $f_1, \ldots, f_s$, see~\cite[Tag 00F2]{Stacks}.  Let $a_j$, where ${1 \le j \le n}$, be elements of $A$ which are all the  coefficients of the polynomials $f_i$, where $1 \le i \le s$, and all the  coefficients of the Laurent polynomial $b$, see~\eqref{eq-b}. Let $B \subset A$
be a subring generated over $\dz$ by  elements $a_j$, where $1 \le j \le n$. Then $B$ is a Noetherian subring. Since elements $f_1, \ldots, f_s$ are from $B[x_0, \ldots, x_m]$, it is easy to see that $\theta(B[x_0, \ldots, x_m]) {\otimes_B} A \simeq \oo(U) $. Since
$\theta(B[x_0, \ldots, x_m]) / \theta(T) \simeq \oplus_{i=0}^{l-1} Bt^i$ and ${\rm Tor}^B_1 (\oplus_{i=0}^{l-1} Bt^i, A) =0$, we obtain that
$\theta(T) {\otimes_B} A \hookrightarrow \theta(B[x_0, \ldots, x_m]) {\otimes_B} A $. Hence we have
 $\theta(T) {\otimes_B} A \simeq t^l\oo(U) $.

We consider the subring
 $B[{b_0}, x_1, \ldots, x_m]  \subset B [x_0, x_0^{-1}, x_1, \ldots, x_m]$, where the element
 ${b_0}$ is from $B[x_0, x_0^{-1}] $ such that ${\theta({b_0}) =b}$.   From~\eqref{dec-mod2} and above reasonings     we obtain
\begin{equation}  \label{isom-fpr}
\theta (B[{b_0}, x_1, \ldots, x_m]) \otimes_B  A \simeq \oo(U')
\end{equation}
 The  polynomial ring $B[z_0, \ldots, z_m]$ is a Noetherian ring. Therefore
 the kernel of the homomorphism  which is the composition
 $$B[z_0, \ldots, z_m] \stackrel{\tau}{\lrto}  B[b_0, x_1, \ldots, x_m]  \stackrel{\theta}{\lrto} \oo(U \setminus p)  \, \mbox{,}$$
  where $\tau(z_0) = b_0 $, $\tau(z_i) = x_i$ for $1 \le i \le m$,
  is a finitely generated ideal. Hence and from~\eqref{isom-fpr} we obtain that  the ring homomorphism $A \to \oo(U')$ is of finite presentation.

\smallskip

Since $A \to \oo(U')$ is of finite presentation and $p'$ is an effective Cartier divisor on $U'$, the morphism $U'  \to \Spec A$ is smooth near $p'$,
see Remark~\ref{Cartier_smooth}. Since $\oo(U)$ is smooth $A$-algebra and  $\oo(U' \setminus p') = \oo(U \setminus p)$, the $A$-algebra $\oo(U')$
is also smooth, see, e.g, ~\cite[Tag 01V4]{Stacks}.

\medskip
{\em Independence on the choice of $U$.}

Now we prove that the construction of $C'$ does not depend on the choice of $U$ (up to a canonical isomorphism). By means of subdividing of $\Spec A$ and taking a smaller open affine subset of $\Spec A$ if necessary, we can suppose that we change the open set $U$ to the affine open subset
$$D_{U, f} = \{ x \in U \mid  f(x)  \ne 0 \} = \Spec \, \oo(U)_f$$
of $U$, where      $f \in \oo(U)$
such that $f \mod t \oo(U)  \in A^*$.  By~\eqref{declg}  we have decomposition $\vp(\gamma (f) )  = r_- r_+$, where $r_- \in  \vmi(A) $ and $r_+ \in \GG_m(A) \times  \vpl(A)$. As above, we have that elements $r_-$, $r_-^{-1}$ and $\vp^{-1} (r_-^{-1})$ belong to the ring $A[t, t^{-1}] \subset \oo(U \setminus p)$.  Let $g  = \vp^{-1} (r_-^{-1}) f$ be an element from $\oo(U \setminus p)$.   We have that $D_{ U, f} \setminus p = D_{U \setminus p, g} $  and $\vp(\gamma(g)) =r_+$. Besides, $r_+ A[[t]]= A[[t]]$.
Therefore, using~\eqref{cartesian} or~\eqref{detai}   written also for $D_{U, f}$, we have that  by changing of $U$ to $D_{U, f}$ we obtain the changing of $U'$ to $D_{U', g}$, which is an affine open subset. (It is important that from the explicit kind of $D_{U', g}$ we know that it will be an open subset.) Therefore  the construction of $C'$ does not depend on the choice of open $U$ (up to a canonical isomorphism).

\medskip
{\em Further properties.}

We have already checked that the construction of $\oo(U')$ commutes with  base change from $A$ to another commutative ring. Hence it is easy to see that  the construction of $(C', p', \ff', t', e')$ commutes with  base change.

From~\eqref{dec-mod2}, \eqref{eq-b}, \eqref{decaut} (all the coefficients of $b$ except for the coefficient at $t$ are nilpotent elements from $A$)  it follows that $C' \times_A A_{\rm red} = C \times_A A_{\rm red}$, where ${A_{\rm red} = A / {\rm Nil} A}$. Hence this construction does not change  the reduced structures of the schemes, i.e.   ${C'_{\rm red} = C_{\rm red}}$. Hence we obtain that for any scheme $Z$ over $A$ we have that
$$(C \times_{\Spec A} Z)_{\rm red} =  (C' \times_{\Spec A} Z)_{\rm red} \mbox{.}$$
It implies that the topological spaces of schemes $C \times_{\Spec A} Z$ and
$C' \times_{\Spec A} Z$ coincide. In particularly, $C$ and $C'$ has the same topological spaces. From these reasonings it follows also that geometric fibres of $C \to \Spec A$ do not change under this construction.

The morphism $C \to \Spec A$ is separated. This is equivalent to the property that  the diagonal  is a closed subset in the topological space of  $C \times_{\Spec A} C$, see, e.g., \cite[Tag 01KH]{Stacks}.  From above reasonings we have the same condition for  the diagonal subset in
$C' \times_{\Spec A} C'$. Therefore the morphism  $C' \to \Spec A$ is also separated.

Suppose that the morphism $C \to \Spec A$ is proper. Then this morphism is universally closed. From the properties of the construction described above we have that to prove that $C' \to \Spec A$ is proper it is enough to prove that this morphism is universally closed. But this follows, because the topological spaces of $C \times_{\Spec A} Z$ and
$C' \times_{\Spec A} Z$ coincide.

\medskip

{\em Construction and properties of $(C'', p'', \ff'', t'', e'')$.}

We note that $\ff' \mid_{U'} \simeq  \oo_{U'}$.
We have the natural homomorphisms of $\oo(U')$-modules
$$
\ff'(U') \lrto \hat{\ff}_{C', p'}  \xrightarrow{\sim} \hat{\oo}_{C',p'}   \mbox{,}  \qquad     \beta \, :  \, \ff'(U' \setminus p') \lrto  {{\mathcal K }_{C', p'}}
$$
that are constructed by means of the formal trivialization $e'$.
These homomorphisms  lead to the following Cartesian diagram of $\oo(U')$-modules
\begin{equation}  \label{cart-sh}
 {\xymatrix{
   {\ff'(U')} \, \ar[d] \ar@{^{(}->}[r] &  {\ff'(U' \setminus p')} \ar[d]_{\beta}  \\
   {  \hat{\oo}_{C',p'} } \, \ar@{^{(}->}[r]   & {{ {\mathcal K }_{C', p'}  }}
 }}
 \end{equation}
Now we define $C''=C'$, $p''=p'$, $t''=t'$  and the sheaf $\ff''$ by gluing the sheaf $\ff'' \mid_{U'}  \simeq \oo_{U'}$   with the sheaf $\ff' \mid_{C' \setminus p'}$ over $U' \setminus p'$, where the gluing data is uniquely defined by the fibered product of $\oo(U')$-modules
\begin{equation}  \label{diag-sheaf}
 {\xymatrix{
   {\ff''(U')} \, \ar[d] \ar@{^{(}->}[r] &  {\ff'(U' \setminus p')} \ar[d]_{h \, \beta}  \\
   {\hat{\oo}_{C',p'}   \, \ar@{^{(}->}[r]}   & {{\mathcal K }_{C', p'}}
 }}
 \end{equation}
Here $h \, \beta$ is the composition of the map $\beta$ and the map ${ \mathcal K }_{C', p'}  \lrto  {\mathcal K }_{C', p'}$ given by the multiplication on $h \in  L \GG_m(A) \simeq A((t'))^*$. We define  $e''$  as the preimage of $1$ under the isomorphism which is the completion of the map  ${\ff''(U')}  \to {\hat{\oo}_{C',p'}}$ (that was constructed by $e'$).

To check that this construction is well-defined we   choose a relative formal parameter $t'$ at $p'$ from $\oo(U')$ (as we did it in the previous reasonings) and use the decompositions~\eqref{declg} and \eqref{dec} for $h$. According to these decompositions we have $h = h_+ h'$, where ${h_+ \in \GG_m \times \vpl(A)}$, and ${h'  \in \vmi(A) \times \underline{\dz}(A)}$.

We have that $h' \in A[t',t'^{-1}]  \subset \oo(U' \setminus p')$. Besides, from~\eqref{diag-sheaf} we have that
$$
\ff''(U') = \left\{ z \mid z \in \ff'(U' \setminus p') \quad \mbox{such that} \quad h' \beta (z) \in A[[t']]     \right\}  \mbox{.}
$$
Hence  $\ff''(U') = h'^{-1} \ff'(U')  \subset \ff'(U' \setminus p')$, since if $h' \beta (z) \in A[[t']]$  for $z \in \ff'(U' \setminus p')$, then
$h' z \in \ff'(U' \setminus p') $ and $\beta(h'  z) \in A[[t']]$, and from~\eqref{cart-sh} we have that $h'  z  \in \ff'(U')$.

Now to check the independence of the construction on  the choice of $U'$ it is enough to check it  when we change $U'$ to the affine open subset  $D_{U', f}$, where $f \in \oo(U')$
such that $f \mod t' \oo(U')  \in A^*$ (after subdividing of $\Spec A$ and taking a smaller open affine subset of $\Spec A$ if necessary).
In this case we obtain that $\ff'(U')$ is changed to $\ff'(D_{U', f}) = \ff'(U')_f $, and $\ff''(U')$ is changed to  $\ff''(D_{U', f}) = \ff''(U')_f $. The construction maps $\ff'(U')_f $ to $\ff''(U')_f $.

 So, the construction maps quintets to quintets and proper quintets to proper quintets.

It is also easy to see that we have  an action of the group functor $\G$ (which is represented by the group ind-scheme) on the moduli stack $\mm$ and the moduli stack $\mm_{\rm pr}$.

\bigskip

{\em 2. We prove item 2 of Theorem~\ref{action}.}

Let $(C, p, \ff, t, e)$ be a proper quintet
 with a smooth curve $C$ over the field  $k$. Let $(\tilde{C}, \tilde{p}, \tilde{\ff}, \tilde{t}, \tilde{e})$ be a proper quintet over $A$,
 where $A = k[\ve]/ (\ve^2)$, such that its  reduction to $k$ be the quintet $(C, p, \ff, t, e)$.

The affine scheme $(C \setminus p) \times_k (k[\ve]/(\ve^2))$ is smooth and, consequently, formally smooth over  $k[\ve]/(\ve^2)$, and the scheme $\tilde{C} \setminus \tilde{p}$ is also affine.
Therefore by the definition of formally smooth morphism there is a morphism $\kappa$ such that the following diagram is commutative
$$
 {\xymatrix{
   { (C \setminus p) \times_k (k[\ve]/(\ve^2)) }  \ar[d]  &  {C \setminus p}  \ar[l]\ar[d]  \\
   {\Spec k[\ve]/ (\ve^2)}     & {\tilde{C} \setminus \tilde{p}}  \ar[l]  \ar[ul]_{\kappa}
 }}
 $$

From the flatness property of $\tilde{C}$ it is easy to see that $\kappa$ is an isomorphism. Using $\kappa$ and two above quintets, we have two embeddings of the ring $\oo(\tilde{C} \setminus \tilde{p})$
into the ring $(k[\ve]/ (\ve^2)) ((t))$ which are differ by the element from~$\AutL (A)$ (the constructed isomorphism is continuous, since it depends only on finite number of elements from the ring $\oo(\tilde{C} \setminus \tilde{p})$), see~\cite[\S~17.3.4]{FB}. Analogously we obtain an element from
$({\mathcal K }_{C, p} \otimes_k (k[\ve]/ (\ve^2)))^*$ when compare two invertible sheaves from quintets over $k[\ve]/ (\ve^2)$ which have the same curves,  points and  relative formal parameters and  coincide after reduction to $k$.

The constructed elements are from the corresponding Lie algebras.
This gives an element from $\Lie \G(k)$, which maps the quintet obtained as the base change from $k$ to $k[\ve]/(\ve^2)$ of the quintet $(C, p, \ff, t, e)$   to the quintet $(\tilde{C}, \tilde{p}, \tilde{\ff}, \tilde{t}, \tilde{e})$.
\end{proof}

\begin{nt} \em  We note that $\GG_m \subset L\GG_m$ (see~\eqref{md}). It is easy to see that via this embedding we have that  $\GG_m $ is a normal group subfunctor of $\G$, and for any commutative ring $A$ any element from $\GG_m(A)$ maps any quintet over $A$ to an isomorphic quintet over~$A$.
\end{nt}

\begin{nt}
\em
Analog of item~2 of~Theorem~\ref{action} in complex analytic category was given in Proposition~3.19 of~\cite{ADKP}.
\end{nt}

\begin{nt} \em
It was proved in~\cite[\S~3]{ADKP} that in complex analytic category some analog of functor~${\mathcal Q}_{\rm pr}$ (when all fibers of $C$ from a quintet are smooth projective curves of genus~$g$)
is represented by an infinite-dimensional complex manifold.
\end{nt}

\subsection{Determinant line bundle on the moduli stack  of proper quintets  $\mm_{\rm pr}$}  \label{det-lin-bund}

There is a natural notion of a quasicoherent sheaf on a category fibered in groupoids over the category of schemes, see, e.g., Definition~2.1 from
Chapter XIII, \S~2 of~\cite{ACG}.

In particularly, the moduli stack of proper quintets $\mm_{\rm pr}$ is a category fibered in groupoids. To define a  linear bundle on $\mm_{\rm pr}$ it is enough to define  a linear bundle on $\Spec A$ for every proper quintet over any commutative ring $A$ such that the corresponding  compatibilty conditions between these linear bundles are satisfied.

\begin{prop}  \label{cohom}
Let   $C \to \Spec A$ be a family of curves and $p \in C(A)$ such that $C, p$ satisfies the corresponding conditions
from the definition of proper quintet (see Definition~\ref{def-quint}).  Let $\ff$ be a locally free sheaf of $\oo_C$-modules.
Then there is an integer $n >0$ such that the following properties are satisfied.
\begin{enumerate}
\item \label{item1} For any $m \ge n $ the complex
\begin{equation}   \label{compl-AG}
H^0(C, \ff(mp))  \lrto H^0(C, \ff(mp)/\ff )
\end{equation}
is a complex of finitely generated projective $A$-modules whose cohomology groups coincide with the cohomology groups $H^*(C, \ff)$.
    \item
    \label{item2}
    For any $m \ge n $, for any   Cartier divisor $\tilde{p} $ on $C$ such that
 $\oo(mp) \subset \oo(\tilde{p}) $ we have an isomorphism
    $$
    H^0(C, \ff(\tilde{p})) / H^0(C, \ff(mp))  \xrightarrow{\sim}  H^0(C, \ff(\tilde{p})/\ff(mp))  \, \mbox{.}
    $$
\item \label{item3}
Suppose that $(C, p , \ff)$ is a part of a proper quintet  $(C, p, \ff, t, e)$. Then
for any  $m \ge n $ complex~\eqref{compl-AG} is quasi-isomorphic to the complex
\begin{equation}  \label{kr-compl}
 \ff (C \setminus p)  \lrto  A((t))/A[[t]]  \, \mbox{,}
\end{equation}
 where the map in this complex is induced by the formal trivialization $e$,  isomorphism~\eqref{K-isom} and the fact that $$\ff (C \setminus p) =  \varinjlim_{m  > 0}   H^0(C, \ff (mp)) \mbox{.}$$
 Besides, for any $k_1 \le k_2$ we have canonical isomorphisms of $A$-modules
 \begin{equation}  \label{quot-isom}
 H^0(C, \ff(k_2)/\ff(k_1)) \xrightarrow{\sim} t^{-k_2}A[[t]]/ t^{-k_1} A[[t]] \, \mbox{.}
 \end{equation}
\end{enumerate}
\end{prop}
\begin{proof}
1. We note that the restriction of $p$ to any fiber of $  C \to \Spec A$ is an ample divisor on a curve. Therefore by~\cite[Corollaire~9.6.4]{EGA-IV-3}
$\oo(p)$ is an ample sheaf on $C$. Therefore by~\cite[III, Corollaire 2.3]{SGA6}, \cite[Tag 0695]{Stacks}  (see also Theorem~A6 from~\cite{BE} and the explanations  before this theorem)  there is an integer $n > 0$ such that $H^q(C, \ff(mp)) =0 $ for any $q > 0$ and any $m \ge n $.

By \cite[Theorem~2.9]{Kie}, \cite[Tag 0B91]{Stacks} (see also Theorem~A7 from~\cite{BE})   $H^0(C, \ff(mp))$ is a pseudo-coherent, and even  perfect, $A$-module. Hence  $H^0(C, \ff(mp))$  is a finitely presented $A$-module.  The scheme $C$ is quasicompact and separated. Therefore we can write
the \v{C}ech complex for the sheaf $\ff(mp)$  and a  finite affine covering of $C$, where the all intersections will be also affine open subsets. This complex is of finite length and all of the terms of this complex are flat $A$-modules. Since $H^q(C, \ff(mp)) =0 $ for any $q > 0$, this complex is acyclic in positive degrees, see~\cite[Tag~01XD]{Stacks}. Therefore $H^0(C, \ff(mp))$ is a flat $A$-module. Since this module is also finitely presented, it is a finitely generated projective $A$-module.

The $A$-module $H^0(C, \ff(mp)/\ff)= H^0(p, \ff(mp)/\ff)$ is a finitely-generated projective $A$-module, since it is an extension of $A$-modules $H^0(p, \ff((k+1)p)/ \ff(kp)$ of the same  type.

Now the statement on the cohomology groups $H^*(C, \ff)$ follows from the long exact cohomological sequence associated with the exact sequence of sheaves:
$$
0 \lrto \ff \lrto \ff(mp) \lrto \ff(mp)/\ff  \lrto 0 \, \mbox{.}
$$

2. The statement of  item~\ref{item2} follows from
$H^1(C, \ff(mp)) =0$ and the long exact cohomological sequence associated with exact sequence of sheaves
$$
0 \lrto \ff(mp) \lrto \ff(\tilde{p})  \lrto \ff(\tilde{p})/\ff(mp)  \lrto 0  \, \mbox{.}
$$

3. Complex~\eqref{kr-compl} is an inductive limit of complexes~\eqref{compl-AG}  over all integers $m \ge n$.
Isomorphisms~\eqref{quot-isom} follow from the existence of $t$ and $e$ in the definition of proper quintet.

\end{proof}
\begin{nt} \em
The complex~\eqref{compl-AG} is the complex (4.15) from Chapter~XIII, \S~4 of~\cite{ACG}. We added the proofs and the corresponding references when the ring $A$ is non-Noetherian.
\end{nt}
\begin{nt} \em
The complex~\eqref{kr-compl} is quasi-isomorphic to the complex
$$
\ff (C \setminus p) \oplus A[[t]] \lrto  A((t)) \, \mbox{.}
$$
When $A$ is a Noetherian ring, the statement that the cohomology groups of the last complex coincide with $H^*(X, \ff)$ is a consequence of Theorem~2 from~\cite{Osip}, since  $p$ is an ample divisor (see the proof of Proposition~\ref{cohom}), and therefore $C \setminus p$ is an affine open subset.
\end{nt}
\begin{nt}  \em
It follows rom the proof of Proposition~\ref{cohom}  that the statement in item~\ref{item3} about   complex~\eqref{kr-compl} remains true when we  change $A((t))/ A[[t]]$  to  $\varinjlim_{m  > 0}   H^0(C, \ff (mp)/ \ff)$ and  take $C$ and $p$ as in the proposition, and $\ff$  any locally free sheaf of $\oo_C$-modules. Besides, we do not demand an existence of $t$ and $e$.
\end{nt}

\bigskip
By item~\ref{item1} of Proposition~\ref{cohom}  for any proper quintet $(C, p, \ff, t, e)$ over $A$, for any  integer $m $ (as in the proposition) we have the projective $A$-module of rank $1$
\begin{equation}  \label{det-bund}
\Hom\nolimits_A \left( \bigwedge^{\rm max}  H^0(C, \ff(mp)/\ff ),  \bigwedge^{\rm max} H^0(C, \ff(mp))  \right)  \, \mbox{.}
\end{equation}
By item~\ref{item2} of Proposition~\ref{cohom}  this $A$-module does not depend on the choice of $m$ up to a canonical isomorphism (or we can choose the minimal $m$ that we need in Proposition~\ref{cohom}), and this $A$-module is compatible with  base change such that the obvious diagram of compatibility is satisfied. This $A$-module is the determinant of the cohomology of $\ff$, which was denoted by $\det R \pi_* \ff$ in~\S~\ref{Deligne},
where $\pi \, : \, C \to \Spec A$ is the morphism from the definition of the quintet.

Thus by a proper quintet $(C, p, \ff, t, e)$ over $A$ we have  constructed a projective $A$-module of rank $1$, which is the determinant of the cohomology of $\ff$. This defines a linear bundle on $\mm_{\rm pr}$ which we call {\em the determinant line bundle on $\mm_{\rm pr}$}.

\bigskip

A linear bundle on a stack (or on a category fibered in groupoids over the category of schemes) can be considered also as a category fibered in groupoids over the category of schemes with the natural functor between two categories fibered in groupoids.
 Therefore when a group functor from the category of schemes (which will be in our case the
sheaf of groups that is represented by an ind-scheme) acts on a stack,  we can speak about  an action of this  functor on a line bundle on the stack
 which lifts the action on the stack,
  see, e.g.,~\cite[\S~1]{R}.

\begin{Th}  \label{act-det-st}
  There is a natural action of the group ind-scheme $\widetilde{\G}$ on the determinant line bundle on the moduli stack $\mm_{\rm pr}$ that lifts an action of the group ind-scheme $\G$ on $\mm_{\rm pr}$ constructed in Theorem~\ref{action}.
\end{Th}
\begin{proof}
Let $(C,p,\ff, t, e)$ be a quintet over $A$ with $\pi : C \to \Spec A$.
The determinant line bundle on $\Spec A$ that corresponds to this quintet is determined  by a projective $A$-module  $\det R \pi_* \ff$ (see formula~\eqref{det-bund}).
Let
$\upsilon \, : \, \ff(C \setminus p)  \to A((t))$
 be the natural map (see item~\ref{item3} of Proposition~\ref{cohom}).

Let an element  $g$ be from $\G(A)$.
We denote the new quintet over $A$ after the  action by $g$ as
$(g(C)), g(p), g (\ff), g(t), g(e)) $ with $g(\pi) : g(C) \to \Spec A$. The determinant line bundle on $\Spec A$ that corresponds to the new quintet is determined by a projective $A$-module  $\det R g(\pi)_* g(\ff)$.

Since $g(A((t))) = A((t))$, after the action by $g$ we have that complex~\eqref{kr-compl} for the quintet $(C, p, \ff, g, e)$  is isomorphic to the following complex:
\begin{equation}  \label{compl-C}
\ff(C \setminus p) \stackrel{\Theta_1}{\lrto} A((t))/ g(A[[t]])  \, \mbox{,}
\end{equation}
where the map $\Theta_1$ is the composition of the map $g \upsilon$ with the quotient map from $A((t))$ to $A((t))/ g (A[[t]])$ (here and further $g \upsilon$ is the composition of $\upsilon$ and the action by $g$ on $A((t))$).

From the construction of the quintet
$
(g(C)), g(p), g (\ff), g(t), g(e))
$,
since $$g(\ff)(g(C) \setminus g(p))  \simeq \ff(C \setminus p) \, \mbox{,} $$
we have that complex~\eqref{kr-compl} written for the quintet
$
(g(C)), g(p), g (\ff), g(t), g(e))
$
is isomorphic to the complex
\begin{equation}  \label{compl-B}
\ff(C \setminus p)  \stackrel{\Theta_2}{\lrto} A((t))/ A[[t]]  \, \mbox{,}
\end{equation}
where the map $\Theta_2$ is the composition of the map $g \upsilon$ with the quotient map from $A((t))$ to $A((t))/ A[[t]]$.

From~\cite[Prop.~3.2]{O1} and Proposition~\ref{cohom} it is easy ro see  that there is an integer $r > 0$ such that  $g(t^{-r}A[[t]])  \supset A[[t]]$ and complex~\eqref{compl-C} contains a quasi-isomorphic subcomplex of finitely generated projective $A$-modules
\begin{equation}  \label{com1}
H^0(C, \ff(rp))  \stackrel{\Theta_1}{\lrto} g(t^{-r}A[[t]])/ g(A[[t]])
\end{equation}
that is isomorphic to a complex~\eqref{compl-AG},
and complex~\eqref{compl-B} contains a quasi-isomorphic subcomplex of finitely generated projective $A$-modules
 \begin{equation}  \label{com2}
H^0(C, \ff(rp))  \stackrel{\Theta_2}{\lrto} g(t^{-r}A[[t]])/ A[[t]]
\end{equation}
that contains a quasi-isomorphic subcomplex~\eqref{compl-AG} written for the quintet \linebreak $(g(C), g(p), g (\ff), g(t), g(e))$.

From complexes~\eqref{com1}-\eqref{com2} we obtain that there are canonical isomorphisms
\begin{multline*}
\det R \pi_* \ff \simeq  \Hom\nolimits_A  \left( \bigwedge^{\rm max} ( g(t^{-r}A[[t]])/ g( A[[t]])), \bigwedge^{\rm max} H^0(C, \ff(rp)   \right) = \\ =
\Hom\nolimits_A \left(\det  ( g(A[[t]])  \mid   g(t^{-r}A[[t]])), \bigwedge^{\rm max} H^0(C,  \ff(rp))        \right)  \mbox{,}
\end{multline*}

\begin{multline*}
\det R g(\pi)_* g(\ff) \simeq  \Hom\nolimits_A  \left( \bigwedge^{\rm max} (g(t^{-r}A[[t]])/  A[[t]]), \bigwedge^{\rm max} H^0(C, \ff(rp))   \right) = \\ =
\Hom\nolimits_A \left(\det  ( A[[t]]  \mid   g(t^{-r}A[[t]])), \bigwedge^{\rm max} H^0(C,  \ff(rp))        \right) \mbox{.}
\end{multline*}

From~\eqref{iso-tens} we have
$$
\det  ( g(A[[t]])  \mid   g(t^{-r}A[[t]])) \simeq  \det(g(A[[t]])  \mid A[[t]])  \otimes_A   \det  ( A[[t]]  \mid   g(t^{-r}A[[t]])) \, \mbox{.}
$$

Hence we  obtain a canonical isomorphism
\begin{equation}  \label{det-isom}
\det R \pi_* \ff \otimes_A     \det(g(A[[t]])  \mid A[[t]])    \xrightarrow{\sim}  \det R g(\pi)_* g(\ff)
\end{equation}
such that the following diagram of isomorphisms is commutative for any $g_1$ and $g_2$ from~$\G(A)$

$$
 {\xymatrix{
   \det R g_1 g_2(\pi)_* g_1g_2(\ff)
    &
    \det R \pi_* (\ff) \otimes_A \det(g_1 g_2 (A[[t]])  \mid  A[[t]]) \ar[l] \\
    \det R  g_2(\pi)_* g_2(\ff)  \otimes_A  \det(g_1 ( A[[t]])   \mid   A[[t]]) \ar[u] &
   P  \ar[l] \ar[u]
    }}
 $$
where
$$ P = \det  R  \pi_* \ff  \otimes_A   \det(g_1 g_2 ( A[[t]])   \mid  g_1(A[[t]]))
       \otimes_A \det(g_1 ( A[[t]])   \mid A[[t]]) \, \mbox{.}$$

Now using the definition of~$\widetilde{\G}$ from \S~\ref{centr-det-g}  and the compatibility of isomorphism~\eqref{det-isom}  with  base change, we finish the proof of the theorem.
\end{proof}

\section{$2$-cocycle on $\G$ obtained from $\cup$-products of $1$-cocyles}

\label{2-cocycles}

We will construct a $2$-cocycle on $\G$ by means of products of   $\cup$-products of $1$-cocyles on $\G$ and the Contou-Carr\`{e}re symbol $\CC$.
We also calculate the corresponding  Lie algebra $2$-cocycle on $\Lie \G$.

\subsection{Contou-Carr\`{e}re symbol}  \label{CC-symb}

We recall the definition of the Contou-Carr\`{e}re symbol (see~\cite{CC1}, \cite[\S~2.9]{D2}, \cite[\S~2]{OZ}):
$$
\CC \,  :  \, L\GG_m  \times L \GG_m \lrto L \GG_m \otimes L \GG_m \lrto \GG_m  \, \mbox{.}
$$

Let $A$ be any commutative ring and  $\underline{n}  \in \underline{\dz}(A) $.
The element $\underline{n}$ defines a decomposition $A = A_1 \times \ldots \times A_k$
into the finite direct product of rings such that $\underline{n}$ restricted to every $\Spec A_l$ equals a constant function with value $n_l \in \dz$.
Now
for any $a \in A^*$
we define the element $a^{\underline{n}} = a^{n_1} \times \ldots \times a^{n_k}$ from $ A^*$.

We define a  free $A((t))$-module of rang $1$:
\begin{equation}  \label{form-ker}
\widetilde{\Omega}^1_{A((t))}  = \Omega^1_{A((t))}  / N  \, \mbox{,}
\end{equation}
where $\Omega^1_{A((t))}$ is the $A((t))$-module of absolute K\"ahler differentials, the  $A((t))$-submodule $N$ is generated by all elements
 $df -  f' dt$, where $f \in A((t))$ and $f'= \frac{\partial{f}}{\partial{t}}$. It is clear  that $N$ contains elements $da$, where $a \in A$, and  that $dt$ is a basis of the $A((t))$-module~$\widetilde{\Omega}^1_{A((t))}$.

We define the  residue
$$
\res \; : \;  \Omega^1_{A((t))}  \lrto \widetilde{\Omega}^1_{A((t))}   \lrto A
$$
 where the first map is the natural map, and the second map is the map $\sum_{i \in \dz} a_it^i dt   \mapsto a_{-1}$.

The  Contou-Carr\`{e}re symbol $\CC$ is a bimultiplicative antisymmetric morphism. It has the following additional properties. Let  ${f,g  \in A((t))^*}$.

\begin{enumerate}
\item  If $a \in A^*$, then ${\CC}(a, g ) = a^{\nu(g)}$.
\item   $\CC(t,t)=-1$.
\item  If $\Q \subset A$ and $ f \in \vpl(A) \times \vmi(A) $, then
\begin{equation}  \label{CC-exp-log-1}
\CC(f,g) = \exp \res \left(\log f \cdot \frac{dg}{g} \right)  \, \mbox{,}
\end{equation}
where  $\exp (x)$ and $\log(1+y)$ are the usual formal series, the series $\log$ in above formula  converges in the topology of $A(((t))$,
and application of series $\exp$ in above formula makes sense, because  $\res \left(\log f \cdot \frac{dg}{g} \right)  \in {\rm Nil }(A)$.
\end{enumerate}

After restriction to commutative $\Q$-algebras,  the Contou-Carr\`{e}re symbol $${\CC}_{\Q} \, : \,  {L\GG_m}_{\Q}  \times {L \GG_m}_{\Q}  \lrto {\GG_m}_{\Q} $$
 is uniquely defined by the above properties.

For an arbitrary $A$, any $f, g \in A((t))^*$ there are  unique decompositions:
$$
f = \prod_{i < 0} (1 - a_it^i) \cdot a_0 \cdot t^{\nu(f)}  \cdot \prod_{i > 0} (1 - a_it^i) \,
\mbox{,}  \qquad
g = \prod_{j < 0} (1 - b_jt^j) \cdot b_0 \cdot t^{\nu(g)}  \cdot \prod_{j > 0} (1 - b_jt^j)   \, \mbox{,}
$$
where $a_0, b_0 \in A^*$, $a_i, b_j \in {\rm Nil}(A) $ when $i, j <0$, and the products over negative $i$ and over negative $j$ are finite products.
Now the Contou-Carr\`{e}re symbol is defined by  the following formula:
\begin{equation}  \label{CC-form-1}
\CC(f,g)= (-1)^{\nu(f)  \nu(g)}  \frac{a_0^{\nu(g)} \prod_{i > 0} \prod_{j>0 }
 \left(1 - a_i^{j /(i,j)} b_{-j}^{i/(i,j)} \right)^{(i,j)}  }{b_0^{\nu(f)}
 \prod_{i > 0} \prod_{j>0 }
 \left(1 - a_{-i}^{j /(i,j)} b_{j}^{i/(i,j)} \right)^{(i,j)}
  }  \, \mbox{,}
\end{equation}
where the products in  the numerator and denominator actually consist of a finite number of factors, therefore the formula makes sense. (Formula~\eqref{CC-form-1} can be obtained, using above properties and by application of formula~\eqref{CC-exp-log-1} to elements $1 - a_it^i$ and $1 - b_jt^j$.)

We recall, see~\eqref{md}, that we have embeddings $\GG_m \times \vpl  \hookrightarrow L\GG_m$
and $\GG_m \times \vmi  \hookrightarrow L\GG_m$. From above formulas for   the  Contou-Carr\`{e}re symbol $\CC$ it follows that
\begin{equation}  \label{CC-ident}
\CC \mid_{(\GG_m \times \vpl) \times (\GG_m \times \vpl)} =1  \qquad \mbox{and} \qquad
\CC \mid_{(\GG_m \times \vmi) \times (\GG_m \times \vmi)} =1
   \mbox{.}
\end{equation}

\bigskip

It is important for application of the  Contou-Carr\`{e}re symbol $\CC$ in this article that $\CC$ is invariant under diagonal action of the group functor $\AutL$
 on the commutative group functor $L\GG_m \times L \GG_m$, see~\cite[\S~2]{OZ}, \cite{GO3}.

\subsection{$\cup$-products of $1$-cocyles on $\G$ and the corresponding Lie algebra $2$-cocycles}
\label{prod-coc}

Since $L\GG_m$ is an $\AutL$-module, from the natural morphism of group functors ${\G  \to \AutL}$
we obtain that
$L\GG_m$ is also a $\G$-module.

 Besides, the   Contou-Carr\`{e}re symbol $\CC$ is a morphism of $\G$-modules \linebreak ${L \GG_m \otimes L \GG_m  \to \GG_m}$,
where we consider the trivial action of $\G$ on $\GG_m$. Hence, the following definition is well-defined.

\begin{defin}
For any two  $1$-cocycles $\lambda_1$ and $\lambda_2$ on the group functor $\G$
with coefficients in the $\G$-module $L \GG_m$ we define  the  $2$-cocycle
$$\langle  \lambda_1 , \lambda_2 \rangle = \CC \circ (\lambda_1  \cup \lambda_2 ) $$
  on the group functor  $\G$ with coefficients in the trivial $\G$-module  $\GG_m$, where
   $\circ$ means the composition of morphisms of functors.
\end{defin}

\begin{defin}  \label{def-1-coc}
We define  $1$-cocycles $\Lambda$ and $\Omega$
 on the group functor  $\G$ with coefficients in the $\G$-module $ L\GG_m$:
 $$
 \Lambda ((h, \vp)) =h \, \mbox{,} \qquad \Omega( (h, \vp) ) = \widetilde{\vp}' = d \vp(t)/ dt \,  \mbox{,}
 $$
where $ (h, \vp) \in \G(A) $ for any commutative ring $A$.
\end{defin}
It is easy to see that Definition~\ref{def-1-coc} is well-defined, i.e.  $\Lambda$ and $\Omega$ are $1$-cocycles.

\begin{nt}  \label{univers}
\em

The $1$-cocycle $\Lambda$ is universal in the following sense. Let $\lambda$ be any $1$-cocycle on $\G$ with values in the $\G$-module  $L\GG_m$.
We define the new morphism of {\em group functors}
$$ \Phi_{\lambda} \, :  \, \G \lrto \G \, \mbox{,}
\qquad \qquad
\Phi_{\lambda} ((h, \vp)) = (\lambda(h, \vp), \vp)  \, \mbox{,}
$$
where $(h, \vp) \in \G(A) $ for any commutative ring $A$. Now under inverse image we have $1$-cocycles $\Phi_{\lambda}^*(\Lambda) = \lambda$,
$\Phi_{\lambda}^* (\Omega) = \Omega$, and correspondingly we have $2$-cocycle ${\Phi_{\lambda}^* (\langle  \lambda_1 , \lambda_2 \rangle) =
\langle  \Phi_{\lambda}^* (\lambda_1) ,   \Phi_{\lambda}^* (\lambda_2)    \rangle }$.
\end{nt}

\bigskip

Any $2$-cocycle  $\Upsilon$ on $\G$ with coefficients in the trivial $\G$-module $\GG_m$ induces the Lie algebra $2$-cocycle (see~\cite[Appendix A.3]{O1})
$$
\Lie \Upsilon \, : \, \Lie \G \times \Lie \G  \lrto \GG_a  \, \mbox{,}
$$
where $\GG_a =\Lie \GG_m $, $\GG_a (A) =A$ for any commutative ring $A$.
We take  elements $d_i $ from $ \Lie \G (A)  \subset  \G (A[\ve_i]/ (\ve_i^2)) $
 and consider these elements as  $d_i \in \G(A[\ve_1, \ve_2] / (\ve_1^2, \ve_2^2))$ through the natural embeddings
 $\G (A[\ve_i]/ (\ve_i^2))  \hookrightarrow \G(A[\ve_1, \ve_2] / (\ve_1^2, \ve_2^2))$, where $i=1$ and $i=2$.
  Now we have that an element
\begin{equation}  \label{Lie-coc}
\Lie \Upsilon (d_1, d_2) = \Upsilon(d_1,d_2) \Upsilon(d_2, d_1)^{-1}  \in  (A[\ve_1, \ve_2] / (\ve_1^2, \ve_2^2))^*
\end{equation}
is the image of the element from $A$ as $a \mapsto 1 + a \ve_1 \ve_2$.

We recall Remark~\ref{tng-Lie} and Proposition~\ref{prop-iso-Lie} for the description of $\Lie \G(A)$. We obtain a proposition.
\begin{prop} \label{Lie-cup}
Let $A$  be any commutative ring. For any elements $s_i, r_i$ from $A((t))$, where $i=1$ and $i=2$, we have explicit expressions for the Lie algebra $2$-cocycles:
\begin{gather}
\label{first-form}
\Lie \langle \Lambda, \Lambda  \rangle \left(s_1 + r_1 \frac{\partial}{\partial t}, s_2 + r_2 \frac{\partial}{\partial t} \right) = 2 \res (s_1 d s_2) \, \mbox{,}\\
\label{sec-form}
\Lie \langle \Lambda, \Omega  \rangle \left(s_1 + r_1 \frac{\partial}{\partial t}, s_2 + r_2 \frac{\partial}{\partial t} \right) =
\res (s_1 d r_2' - s_2 dr_1')  \, \mbox{,} \\
\label{third-form}
 \Lie \langle \Omega, \Omega  \rangle \left(s_1 + r_1 \frac{\partial}{\partial t}, s_2 + r_2 \frac{\partial}{\partial t} \right) = 2 \res (r_1' d r_2')  \, \mbox{.}
\end{gather}
\end{prop}
\begin{proof}
The $2$-cocycle $\langle \Omega, \Omega  \rangle $ is the formal Bott-Thurston cocycle from Proposition~2.3 of~\cite{O1}. Then formula~\eqref{third-form} is the statement of Proposition~4.1 from~\cite{O1}.

Formulas~\eqref{first-form}-\eqref{sec-form}
follow by direct calculations with formulas~\eqref{Lie-coc} and~\eqref{CC-exp-log-1}-\eqref{CC-form-1},     by the similar reasonings as in the proof of Proposition~4.1 from~\cite{O1}. Indeed,
  we use that for any $f_1, f_2$ from $A((t))$ and $i=1$, $i=2$ we have for the  Contou-Carr\`{e}re symbol  $\CC$ in the ring $E((t))$, where $E = A[\ve_1, \ve_2] / (\ve_1^2, \ve_2^2)$, that
 $$
 \CC( 1 + f_1 \ve_i, 1 + f_2 \ve_1 \ve_2) =1  \, \mbox{.}
 $$
 Then
 the calculations in case of formula~\eqref{first-form} lead to the following equalities  using the  ring $E((t))$
\begin{multline*}
\CC(1 + s_1 \ve_1, 1+ s_2 \ve_2) \CC(1 + s_2 \ve_2, 1+ s_1 \ve_1)^{-1}= \\ = (1 + \res (s_1 ds_2) \ve_1 \ve_2 ) (1 + \res (s_2 ds_1) \ve_1 \ve_2 )^{-1}=
1 + 2 \res (s_1 ds_2) \ve_1 \ve_2 \, \mbox{.}
\end{multline*}
And in case of formula~\eqref{first-form} we obtain
\begin{multline*}
\CC(1 + s_1 \ve_1, 1+ r_2' \ve_2) \CC(1 + s_2 \ve_2, 1+ r_1' \ve_1)^{-1}=  \\ = (1 + \res (s_1 dr_2') \ve_1 \ve_2 ) (1 + \res (s_2 dr_1') \ve_1 \ve_2 )^{-1}=
1 +  \res (s_1 dr_2' - s_2 dr_1') \ve_1 \ve_2 \, \mbox{.}
\end{multline*}

\end{proof}

\begin{nt}  \label{Lie-dim} \em
Let $ A=k$ be a field of characteristic zero. By similar reasonings as in the proof of Proposition~(2.1) from~\cite{ADKP} we can see that the continuous Lie algebra cohomology $H^2_{\rm c}(\Lie \G(k), k)$ is a three-dimensional vector space over $k$ with the basis given by the {$2$-cocycles} from
formulas~\eqref{first-form}-\eqref{third-form}.
\end{nt}

\section{Comparison of central extensions of $\G$}

\label{compar}

We will prove a local analog of the Deligne-Riemann-Roch isomorphism~\eqref{DRR1}. This is the comparison of two central extensions of $\G_{\Q}$ by ${\GG_m}_{\Q}$. The first central extension is the determinant central extension of $\G$ from Section~\ref{main-centr}, and the second central extension is given  by some products of $2$-cocycles constructed in Section~\ref{prod-coc}. At first, we prove the corresponding result for the Lie algebra $2$-cocycles. We will use also the theory of infinitesimal formal groups over $\Q$.

\subsection{Comparison of Lie algebra $2$-cocycles}

We recall that by Proposition~\ref{prop-det}  the determinant central extension of $\G^0$ admits a natural section $\G^0  \to \widetilde{\G^0}$  (as functors) with the corresponding explicit $2$-cocycle $D$  given by formula~\eqref{coc-D-2}.

Let $A$ be any commutative ring.
By Proposition~\ref{prop-iso-Lie} any element $z $ from $\Lie \G^0(A)$ naturally acts on the $A$-module $A((t))$ by continuous $A$-endomorphisms. We write this action as a block matrix (cf.~formula~\eqref{matr'})
\begin{equation}
\begin{pmatrix}
a_z & b_z  \\
c_z & d_z
\end{pmatrix}
\end{equation}
with respect to the decomposition $A((t)) = t^{-1} A[t^{-1}]  \oplus A[[t]]$, and where the matrix acts on  elements-columns from $A((t))$
on the left.

The following proposition is an analog of Proposition~4.2  from~\cite{O1} for $\AutL$ and also a formal analog of Proposition 6.6.5 from~\cite{PS} from the theory of smooth loop groups.

\begin{prop}  \label{Lie-D}
For any commutative ring $A$ and any elements $z,w$ from $\Lie \G^0(A)$ we have
\begin{equation} \label{form-Lie-D}
\Lie D (z,w) = \tr (c_w b_z - c_z b_w)  \, \in A \, \mbox{,}
\end{equation}
where the trace of the $A$-linear map $A[[t]] \to A[[t]]$ is well-defined (because there is an integer $n \ge 0$ such that
$b_z \mid_{t^n A[[t]]} = b_w \mid_{t^nA[[t]]]} = 0$).
\end{prop}
\begin{proof}
By formulas~\eqref{coc-D-2} and~\eqref{Lie-coc} we have to calculate over the ring $A[\ve_1, \ve_2]/ (\ve_1^2, \ve_2^2)$:
\begin{multline*}
D (\id + z \ve_1, \id + w \ve_2) D(\id + w \ve_2, \id + z \ve_1)^{-1} = \det (\id - c_z b_w \ve_1 \ve_2)
\det(\id - c_w b_z \ve_1 \ve_2)^{-1} = \\ =
(1 - \tr(c_z b_w) \ve_1 \ve_2 )(1 - \tr(c_w b_z) \ve_1 \ve_2 )^{-1} = 1 + \tr (c_w b_z - c_z b_w) \ve_1 \ve_2  \, \mbox{.}
\end{multline*}
Here $\id$ is the identity map.
\end{proof}

\bigskip

\begin{Th} \label{main-Lie}
We have an equality between Lie algebra $2$-cocycles on $\Lie \G^0$ with coefficients in $\GG_a$
\begin{equation}  \label{Lie-RR}
12 \Lie D = 6 \Lie \langle \Lambda, \Lambda \rangle  - 6 \Lie \langle \Lambda, \Omega \rangle  + \Lie \langle \Omega, \Omega \rangle   \, \mbox{.}
\end{equation}
\end{Th}
\begin{proof}
We will use Propositions~\ref{Lie-cup}  and~\ref{Lie-D}.

We fix a commutative ring $A$. There is a natural topology on the $A$-module $\Lie \G^0(A) = A((t))  + A((t)) \frac{\partial}{\partial t}$.
It is easy to see that $2$-cocycles given by formulas~\eqref{first-form}-\eqref{third-form} are continuous in each argument when we consider the discrete topology on $A$. Besides, $\Lie D$ is also continuous, since   $\tr (c_w b_z)$  and $\tr (c_z b_w)$ are continuous in each argument (see formula~\eqref{form-Lie-D}).

Therefore it is enough to check equality~\eqref{Lie-RR} on elements of type $t^i$ and $t^{j+1} \frac{\partial}{\partial t}$ from $\Lie \G^0(A)$.
We will use the Kronecker delta $\delta_{r,s}$ that is equal to $1$ when $r =s$ and is equal to $0$ otherwise. We have the following cases (using also that Lie algebra cocycles are antisymmetric).

\smallskip

{\em The first case}. We calculate
$$
12 \Lie D \left(t^n, t^m \right) = 12 \tr \left( c_{t^m} b_{t^n} - c_{t^n} b_{t^m}  \right) = 12 \left( \sum_{l=0}^{m-1} 1 \right) \delta_{m, -n} = 12 m \, \delta_{m, -n}  \,  \mbox{.}
$$
Here the result is equal to the first summand  $12 \tr( c_{t^m} b_{t^n})$ when $m > 0$ and it is equal to the second summand $12(- c_{t^n} b_{t^m})$   when $m \le 0$ (the other summands are equal to zero).

The right hand side of~\eqref{Lie-RR} applied to the elements $t^n$ and $t^m$ is equal to
$$
6  \Lie \langle \Lambda, \Lambda \rangle  \left(t^n, t^m\right) = 6 \cdot 2 \res \left(t^n d t^m \right)= 12 m \, \delta_{m, -n}  \,  \mbox{.}
$$

\smallskip

{\em The second case}. Now we calculate
\begin{multline*}
12 \Lie D \left(t^n, t^{m+1} \frac{\partial}{\partial t} \right) = 12 \tr \left( c_{t^{m+1}  \frac{\partial}{\partial t}} b_{t^n} - c_{t^n} b_{t^{m+1}  \frac{\partial}{\partial t}  }  \right) = \\ = 12 \left( \sum_{l=0}^{m-1} (l-m) \right)\delta_{m, -n}  = 6 (-m- m^2) \, \delta_{m, -n}  \,  \mbox{.}
\end{multline*}
Here the result is equal to the first summand  $12 \tr  \left( c_{t^{m+1}  \frac{\partial}{\partial t}} b_{t^n}  \right)$ when $m > 0$ and it is equal to the second summand $12  \left(- c_{t^n} b_{t^{m+1}  \frac{\partial}{\partial t}  }  \right) $   when $m \le 0$ (the other summands are equal to zero).

The right hand side of~\eqref{Lie-RR} applied to the elements $t^n$ and $t^{m+1} \frac{\partial}{\partial t}$ is equal to
$$
-6  \Lie \langle \Lambda, \Omega \rangle  \left(t^n,  t^{m+1} \frac{\partial}{\partial t}  \right) = -6  \res \left(t^n d \left(t^{m+1} \right)' \right)=
 6 (-m -m^2)\, \delta_{m, -n}  \,  \mbox{.}
$$

\smallskip

{\em The third case}. At the end,   we calculate

\begin{multline*}
12 \Lie D \left(  t^{n+1} \frac{\partial}{\partial t}  , t^{m+1} \frac{\partial}{\partial t} \right) = 12 \tr \left( c_{t^{m+1}  \frac{\partial}{\partial t}} b_{ t^{n+1} \frac{\partial}{\partial t}  } - c_{ t^{n+1} \frac{\partial}{\partial t}  } b_{t^{m+1}  \frac{\partial}{\partial t}  }  \right) = \\ = 12 \left( \sum_{l=0}^{m-1} l(l-m) \right)\delta_{m, -n}  = 2 (m- m^3) \, \delta_{m, -n}  \,  \mbox{.}
\end{multline*}
Here the result is equal to the first summand  $12 \tr  \left( c_{t^{m+1}  \frac{\partial}{\partial t}} b_{ t^{n+1}  \frac{\partial}{\partial t} }  \right)$ when $m > 0$ and it is equal to the second summand $12  \left( - c_{ t^{n+1}  \frac{\partial}{\partial t  }} b_{t^{m+1}  \frac{\partial}{\partial t}  }  \right) $   when $m \le 0$ (the other summands are equal to zero).

The right hand side of~\eqref{Lie-RR} applied to the elements $t^{n+1} \frac{\partial}{\partial t}  $ and $t^{m+1} \frac{\partial}{\partial t}$ is equal to
$$
 \Lie \langle \Omega, \Omega \rangle  \left(  t^{n+1} \frac{\partial}{\partial t} ,  t^{m+1} \frac{\partial}{\partial t}  \right) =
  2 \res \left((t^{n+1})' d \left(t^{m+1} \right)' \right)=
 2 (m -m^3)\, \delta_{m, -n}  \,  \mbox{.}
$$

 \end{proof}

\subsection{Correspondence between infinitesimal formal groups and Lie algebras}
\label{infini}

Let $k$ be a field of zero characteristic.

It is well-known that the category of finite-dimensional formal groups over $k$ is equivalent to the category of finite-dimensional Lie algebras over $k$,
see, e.g., Theorem~3 in part~II, chapter~V, \S~6 of~\cite{Serre}.
We will need the generalization of this statement to the infinite-dimensional case.

An {\em infinitesimal  formal group  over $k$} is a group ind-scheme \linebreak ${G = \mbox{``$\varinjlim\limits_{i \in I}$''} \Spec A_i}$
 such that every $A_i$ is a finite-dimensional $\Q$-algebra and
 the corresponding profinite algebra of regular functions $\oo(G)= \mbox{$\varprojlim\limits_{i \in I} $} A_i$ is a local $k$-algebra with the residue field $k$.

By~\cite[ch.~I, \S~1.7, \S~2.14; ch. II, \S~2.2, \S~2.5]{Die} and \cite[Expose $\rm VII_B$, \S~2.7]{SGA3}, the functor $G \mapsto \oo(G)$ is an equivalence between
the category of infinitesimal formal groups over $k$ and  the opposite of the category of commutative   linearly compact local $k$-algebras $A= k \oplus U$ (decomposition as
$k$-vector spaces), where $U$ is an open ideal,
such that the topological dual vector space consisting  of continuous  linear functionals $\Hom^{\rm c}_k (A, k)$ is a (discrete) bialgebra over $k$.

Now, similarly to the finite-dimensional case, by~\cite[Expose $\rm VII_B$, \S~3]{SGA3}, there is an isomorphism of bialgebras
$$
U(\Lie G (k) )  \DistTo   \Hom\nolimits^{\rm c}_k (\oo(G), k)  \, \mbox{,}
$$
where $U(\Lie G (k) )$ is the universal enveloping algebra of the Lie algebra $\Lie G(k)$ over~$k$ with the standard bialgebra structure. Moreover,
this isomorphism is induced by the natural linear  embedding of $\Lie G (k)$ to  $\Hom\nolimits^{\rm c}_k (\oo(G), k)$, and $\Lie G(k)  \subset U(\Lie G (k) )$ is the set of primitive elements of the bialgebra.

These constructions lead to the  theorem  (see~\cite[Expose $\rm VII_B$, \S~3.3.2]{SGA3}) that {\em the functor
$G \mapsto \Lie G (k)$ gives an equivalence of the category of infinitesimal formal groups over  $k$ and the category of Lie
algebras over $k$.}

\subsection{Class of the detrminant central extension of ${\G}_{\Q}$   in the group  $H^2 \left({\G}_{\Q},  {\GG_m}_{\Q}\right)$}

We recall, see Section~\ref{Sec-Loop-Aut}, that we have
$$\GG_m \times \vpl  \hookrightarrow (L\GG_m)^0 \qquad   \mbox{and}  \qquad \AutLp  \hookrightarrow \AutL \, \mbox{.} $$
It is clear that $\GG_m \times \vpl$ is an  $\AutLp$-module.
We consider the group subfunctor
$$
\G_+  = (\GG_m \times \vpl)  \rtimes \AutLp \,  \hookrightarrow  \, \G^0  \, \mbox{.}
$$
It is clear that for any commutative ring $A$ elements of $\G_+(A)$ acts on the $A$-module $A[[t]]$ by continuous automorphisms.

\medskip

We will construct the infinitesimal formal group $\II\G^0_{\Q}$ over $\Q$ from the group ind-affine ind-scheme $\G^0_{\Q}$ (see Section~\ref{infini}).

By definition, for any
commutative $\Q$-algebra $A$, the group
$\II\G^0_{\Q}(A)
$
consists of elements $(h, \vp) \in (L\GG_m)^0(A) \rtimes \AutL (A)$ such that in the decompositions
\begin{equation}
\label{coeff}
h = 1 + \sum_i b_i t^i \qquad \mbox{and} \qquad  \widetilde{\vp}= \vp(t) = t + \sum_i c_i t^j
\end{equation}
 all elements $b_i$ and $c_i$ belong to $\Nil(A)$ and they are equal to zero except for a finite number of elements, cf. the definition of the infinitesimal formal group~$\ff \AutL_{\Q}$ constructed from $\AutL_{\Q}$ in~\cite[\S~5.1]{O1}. (The statement that $\II \G^0_{\Q}(A)$ is a group follows easily from decompositions~\eqref{declg} and~\eqref{decaut}.)

From~\eqref{coeff} it follows that the group functor $\II \G^0_{\Q}(A)$ is represented by an ind-scheme
\begin{equation}  \label{expl-ind}
\mbox{``$\varinjlim\limits_{\{\epsilon_i\}}$''} \Spec \Q [b_i; i \in \dz]/ I_{\{\epsilon_i\}}
\times
\mbox{``$\varinjlim\limits_{\{\epsilon_i\}}$''} \Spec \Q [c_i; i \in \dz]/ I_{\{\epsilon_i\}}
 \, \mbox{,}
\end{equation}
where $\Q [b_i; i \in \dz]$ is the polynomial ring over $\Q$ on a set of variables $b_i$ with $i \in \dz$, and the limit is taken over all the sequences $\{\epsilon_i\}$ with $i \in \dz$ and $\epsilon_i$ are nonnegative integers such that all but finitely many $\epsilon_i$ equal zero, the ideal $I_{\{\epsilon_i\}}$ is generated by elements $b_i^{\epsilon_i +1}$ for all $i \in \dz$. The ring $\Q [b_i; i \in \dz]$ and its  ideals $I_{\{\epsilon_i\}}$ has the same description.

Thus, $\II \G^0_{\Q}$ is an infinitesimal formal group over $\Q$. And we have the natural embedding
$$
\II \G^0_{\Q}  \hookrightarrow \G^0_{\Q}
$$
(which is embedding on $A$-points for any commutative $\Q$-algebra $A$).

It is easy to see that the Lie $\Q$-algebra $  \Lie \II \G^0_{\Q}(\Q)$ has a basis consisting of all elements  $d_n = t^{n+1} \frac{\partial}{\partial t}$
and $e_n = t^n$, where $n \in \dz$, with the following relations
\begin{equation}  \label{equat}
\left[d_n, d_m \right]= (m-n) d_{n+m} \, \mbox{,} \qquad \left[d_n, e_m \right]= m e_{n+m}
 \, \mbox{,} \qquad \left[e_n, e_m \right]= 0  \, \mbox{.}
\end{equation}

\medskip

 We define the infinitesimal formal group $\II{\G_+}_{\Q}$ over $\Q$ in the following way
 $$
 \II{\G_+}_{\Q}(A)  =  \II \G^0_{\Q}(A)  \cap \G_+(A)  \, \mbox{,}
 $$
 where $A$ is any commutative $\Q$-algebra, and the intersection is taken in the group $\G^0(A)$. It is clear that the group functor
$\II{\G_+}_{\Q}$ is represented by infinitesimal formal group over~$\Q$ (we have to take in formula~\eqref{expl-ind}  only indices with $i \ge 0$).

We define the infinitesimal formal group $\II {\GG_m}_{\Q}$ over $\Q$ as
$$\II {\GG_m}_{\Q} (A) = 1 + \Nil(A)  \subset A^*  \, \mbox{,}$$
where $A$ is any commutative $\Q$-algebra. (This formal group is the formal group constructed from the completion of the local ring at the identity element of the algebraic group ${\GG_m}_{\Q}$.)

\begin{lemma}  \label{bef-th}
\begin{enumerate}
\item \label{ii1} Any morphism of group functors $\II{\G_+}_{\Q} \to {\GG_m}_{\Q}$ is trivial.
\item \label{ii2}  The homomorphism of $\Q$-algebras of regular functions $\oo(\G^0_{\Q}  \times \G^0_{\Q})  \to \oo(\II \G^0_{\Q}  \times \II \G^0_{\Q})$ induced by the natural morphism of ind-affine ind-schemes  $\II \G^0_{\Q} \times \II \G^0_{\Q} \to \G^0_{\Q}  \times \G^0_{\Q} $ is an embedding.
\end{enumerate}
\end{lemma}
\begin{proof}
1. The Lie $\Q$-algebra $\Lie \II{\G_+}_{\Q}(\Q)$ is a subalgebra of the Lie $\Q$-algebra $  \Lie \II\G^0_{\Q}(\Q)$. The basis of this subalgebra
consists of all elements $d_n$ and $e_m$, where $n \ge -1$ and $m \ge 0$. Then it is easy to see from relations~\eqref{equat}
that
$$ \left[ \Lie \II{\G_+}_{\Q}(\Q), \Lie \II{\G_+}_{\Q}(\Q) \right] = \Lie \II{\G_+}_{\Q}(\Q)
$$
(or we can use the similar formulas as in the proof of Lemma~\ref{commutant}).

From this fact and since the Lie algebra $\GG_a(\Q) = \Lie \II {\GG_m}_{\Q} (\Q)$ is Abelian, we obtain that any Lie algebra homomorphism
from   $\Lie \II{\G_+}_{\Q}(\Q)$  to $\Lie \II {\GG_m}_{\Q} (\Q)$ is equal to zero.

Now, using for any commutative $\Q$-algebra $A$ the homomorphism $A \to A/ \Nil(A)$, it is easy to see that any morphism of group functors
 $\tau  : \II{\G_+}_{\Q} \to {\GG_m}_{\Q}$ comes form the morphism  of infinitesimal formal groups  $\gamma  :  \II{\G_+}_{\Q} \to \II {\GG_m}_{\Q}$
via the composition with the morphism of group functors $\II {\GG_m}_{\Q} \to {\GG_m}_{\Q}$.
Be Section~\ref{infini}, the morphism $\gamma$ is trivial. Therefore $\tau$ is also trivial.

2. This statement easily follows from the description of the corresponding ind-schemes,
since it is embedding of the $\Q$-algebra of mixture polynomials, Laurent polynomials and formal power series in infinite number of variables
to the $\Q$-algebra of formal power series in infinite number of variables,
 see formula~\eqref{expl-ind} and  Section~\ref{ex-ind-sch} (compare also with the proof of Proposition~\ref{mor-gr}).
\end{proof}

\bigskip

Any element from the group $H^2({\G}_{\Q}, {\GG_m}_{\Q})$, where ${\GG_m}_{\Q}$
is a trivial ${\G}_{\Q}$-module,
defines a central extension of group functors that allow a section (as functors), see Section~\ref{cohom-sec}. Therefore it defines the central extension of the corresponding Lie algebras, and a section of central extensions of group functors gives the section of central extension of Lie algebras. Hence, the group of $2$-cocycles of group functor is mapped to the group of Lie algebra $2$-cocycles, see~\cite[Appendix~A.3]{O1}. Moreover we have the natural homomorphisms
\begin{equation}  \label{gr-Lie}
H^2({\G}_{\Q}, {\GG_m}_{\Q})  \lrto H^2(\G^0_{\Q}, {\GG_m}_{\Q})  \lrto H^2(\Lie \G^0_{\Q}(\Q), \Q  ) \lrto H^2(\Lie \II \G^0_{\Q}(\Q), \Q )  \, \mbox{.}
\end{equation}

\begin{nt}  \em
From Proposition~\ref{Lie-cup}, Remark~\ref{Lie-dim} and~\cite[Prop.~(2.1)]{ADKP} (especially its proof) it follows that
$$ \dim\nolimits_{\Q} H^2(\Lie \II \G^0_{\Q}(\Q), \Q ) =3$$
and the map that is the composition of maps
in~\eqref{gr-Lie} is surjective.
\end{nt}

\begin{Th}  \label{th-comp}
Any element from $H^2( \G^0_{\Q}, {\GG_m}_{\Q})$, where ${\GG_m}_{\Q}$
is a trivial $\G^0_{\Q}$-module, is uniquely defined by its image in $H^2(\Lie \II \G^0_{\Q}(\Q), \Q )$ together with its restriction
to $H^2(  {\G_+}_{\Q}, {\GG_m}_{\Q})$.
\end{Th}
\begin{proof}
It is enough to prove that if a central extension of group functors
\begin{equation}  \label{proof-centr-ext}
1 \lrto {\GG_m}_{\Q}  \lrto H \stackrel{\mu}{\lrto} \G^0_{\Q}  \lrto 1
\end{equation}
that admits a section $p$ of $\mu$ (as functors) is isomorphic to the trivial central extension after restriction ${\G_+}_{\Q}$ and the corresponding Lie algebra central extension of  $\Lie \II\G^0_{\Q}(\Q)$  by $\Q$ is isomorphic to the Lie algebra trivial central extension, then
central extension~\eqref{proof-centr-ext} is also isomorphic to the trivial central extension.

By changing the section $p$ to $p \cdot p(e)^{-1}$, where $e$ is the identity element of the group ind-scheme $\G^0_{\Q}$, we can suppose that $p(e)$
is the identity element of the group ind-scheme~$H$. Hence, the corresponding $2$-cocycle $K$ for~\eqref{proof-centr-ext}, constructed by $p$,  satisfies $K(e,e)=1$, where $1 \in {\GG_m}_{\Q}$ is the identity element.

Therefore, by considering for any commutative $\Q$-algebra $A$ the homomorphism ${A \to A/ \Nil(A)}$, it is easy to see that
$K$ maps $\II \G^0_{\Q}  \times \II \G^0_{\Q}$ to the subfunctor $\II {\GG_m}_{\Q}$ of the functor $ {\GG_m}_{\Q}$. Hence we obtain that central extension~\eqref{proof-centr-ext} restricted to $\II \G^0_{\Q} $ comes via $\II {\GG_m}_{\Q}  \to {\GG_m}_{\Q}$ from a central extension
\begin{equation}  \label{in-c}
1 \lrto  \II {\GG_m}_{\Q}  \lrto  T  \stackrel{\kappa}{\lrto}  \II \G^0_{\Q}  \lrto 1
\end{equation}
that admits a section of $\kappa$ (section as functors). Because of this section, $T \simeq  \II {\GG_m}_{\Q}  \times \II \G^0_{\Q}$ (isomorphism as ind-schemes). Therefore $T$ is an infinitesimal formal group over $\Q$.  By our condition, the corresponding to~\eqref{in-c}    Lie $\Q$-algebra
central extension is trivial, and hence it admits a section-homomorphism from $\Lie \II \G^0_{\Q}(\Q)$ to $\Lie T (\Q)$. Hence, by Section~\ref{infini},
there is the corresponding   section $s$ from  $\II \G^0_{\Q}$ to $T$ that is a morphism of group ind-schemes.

Thus, we have constructed the morphism of group ind-schemes
$$
s \, : \, \II \G^0_{\Q}  \lrto H
$$
such that $\mu s = \id$ (we denoted it by the same letter $s$ as the above morphism from  $\II \G^0_{\Q}$ to~$T$).

By our condition, we have also a morphism of group ind-schemes
$$
r \, : \,  {\G_+}_{\Q}  \lrto H
$$
such that $\mu r = \id$.

We claim that $r \mid_{\II {\G_+}_{\Q}}  = s \mid_{\II {\G_+}_{\Q}} $.
Indeed, $r/s$ is a morphism of group functors from $\II {\G_+}_{\Q}$
to ${\GG_m}_{\Q}$ that is trivial by item~\ref{ii1} of Lemma~\ref{bef-th}.

For any commutative $\Q$-algebra $A$ and any element $(h, \vp)  \in \G^0(A)$, by~\eqref{declg} and~\eqref{decaut}  we have unique decompositions
$$
h = h_- h_+  \, \mbox{,} \qquad \qquad \vp = \vp_+ \vp_- \,  \mbox{,}
$$
 where  $ h_- \in \vmi(A)$,
$h_+ \in \GG_m(A) \times \vpl(A)$,  $ \vp_+ \in \AutLp$, $ \vp_- \in \AutLm$. These decompositions are functorial with respect to $A$.
We define the morphism of functors $q : \G^0_{\Q} \to H$ such that $\mu q =\id$ in the following way
$$
q ((h, \vp)) = s(h_-) r(h_+ ) r(\vp_+) s (\vp_-) = s(h_-) r(h_+ \vp_+) s (\vp_-)  \, \mbox{.}
$$
To prove the theorem it is enough to prove that $q$ is a morphism of group functors, i.e. that it preserves the group structure.

We claim that $q \mid_{\II \G^0_{\Q}} = s$. Indeed, if $(h, \vp) \in \II \G^0_{\Q}(A)$, then
$$ h_+ \vp_+ =  h_-^{-1} (h, \vp) \vp_-^{-1} \, \in  \, \II \G^0_{\Q}(A) \, \mbox{,}$$
since  $h_-^{-1} $ and  $ \vp_-^{-1} $ are from $ \II \G^0_{\Q}(A)$.
Hence, the element $h_+ \vp_+ $ is from the group $ \II {\G_+}_{\Q} (A)$.
Therefore  ${r(h_+ \vp_+) = s (h_+ \vp_+)}$.  Since $s$ is a morphism of group functors,  we obtain
$$
q ( (h, \vp)  ) = s(h_-) r(h_+ \vp_+) s (\vp_-) = s(h_-) s(h_+ \vp_+) s (\vp_-) = s ((h, \vp) )  \, \mbox{.}
$$

Now we consider a morphism of ind-schemes
$$
\beta \, : \, \G^0_{\Q} \times \G^0_{\Q}  \lrto {\GG_m}_{\Q} \, \mbox{,} \qquad \beta(g_1 \times g_2)= q(g_1) q(g_2) q(g_1 g_2)^{-1}  \, \mbox{,} \qquad
g_1, g_2 \in \G^0_{\Q}(A)  \, \mbox{.}
$$
We have to prove that $\beta = \bf{1}$, where ${\bf 1}$ is the constant morphism that is equal to $1$.  We note that $\beta \mid_{\II \G^0_{\Q}  \times \II \G^0_{\Q}} ={\bf 1}$, since $s$ preserves the group structure.
The morphism $\beta$ is determined by the homomorphism of $\Q$-algebras of regular functions
$$
\beta^* \, : \, \oo({\GG_m}_{\Q})  \lrto \oo(\G^0_{\Q} \times \G^0_{\Q})
$$
Using the composition of the morphism  $\II \G^0_{\Q}  \times \II \G^0_{\Q}  \to \G^0_{\Q} \times \G^0_{\Q}$ with the morphism $\beta$ and using item~\ref{ii2} of Lemma~\ref{bef-th}, we obtain $\beta^* = {\bf{1}}^*$. Therefore $\beta = \bf{1}$.
\end{proof}

\medskip

Immediately from~Theorem~\ref{th-comp} and Theorem~\ref{th1}  we obtain the following corollary.

\begin{cons} \label{Cor2}
Any element from $H^2( \G_{\Q}, {\GG_m}_{\Q})$, where ${\GG_m}_{\Q}$
is a trivial ${\G}_{\Q}$-module, is uniquely defined by its image in $H^2(\Lie \II \G^0_{\Q}(\Q), \Q )$ together with its restriction
to $H^2(  {\G_+}_{\Q}, {\GG_m}_{\Q})$.
\end{cons}

\begin{nt} \em
Theorem~\ref{th-comp} is a generalization of the similar Theorem~5.1 from~\cite{O1} when $\G^0$ is changed to $\AutL$ (for the case of the group of orientation-preserving   diffeomorphisms of the circle in the theory of infinite-dimensional Lie groups see Corollary (7.5) from~\cite{Se}).
\end{nt}

We recall that by Remark~\ref{can-sec} the determinant central extension of $\G$ by $\GG_m$ has the natural section that gives the $2$-cocycle $D$.

Now we obtain a local analog of Deligne-Riemann-Roch isomorphism~\eqref{DRR1}.

\begin{Th}  \label{LDRR}
In the group $H^2 (\G_{\Q}, {\GG_m}_{\Q})$, where ${\GG_m}_{\Q}$
is a trivial ${\G}_{\Q}$-module, we have
\begin{equation}  \label{eq-coc-th}
 D^{12} =  \langle \Lambda, \Lambda \rangle^6 \cdot    \langle \Lambda, \Omega \rangle^{-6} \cdot \langle \Omega, \Omega \rangle  \mbox{.}
\end{equation}
\end{Th}
\begin{proof}
We will use Corollary~\ref{Cor2}. By Theorem~\ref{main-Lie} we have the corresponding equality of the Lie $\Q$-algebra $2$-cocycles.
Besides, by~\eqref{CC-ident} the $2$-cocycle given by the right hand side of~\eqref{eq-coc-th} and restricted to ${\G_+} $ is the constant $2$-cocycle that is equal to $1$. And by Remark~\ref{can-sec}  and formula~\eqref{coc-D-2}, the $2$-cocycle   $D $   restricted to $\G_+$ is also the constant $2$-cocycle that is equal to $1$.
\end{proof}

\vspace{0.5cm}

\noindent Steklov Mathematical Institute of Russsian Academy of Sciences, 8 Gubkina St., Moscow 119991, Russia.

\noindent {\it E-mail:}  ${d}_{-} osipov@mi{-}ras.ru$

\end{document}